\documentclass[a4paper,reqno]{amsart} 

\usepackage{amssymb}
\usepackage[mathcal]{euscript} 

\newcommand{\bydef}{:=}

\newcommand{\id}{\mathrm{id}}

\newcommand{\tr}{\mathrm{tr}}

\newcommand{\bi}{\mathbf{i}}

\newcommand{\cC}{\mathcal{C}}

\newcommand{\cQ}{\mathcal{Q}}

\newcommand{\ZZ}{\mathbb{Z}}

\newcommand{\FF}{\mathbb{F}} 
\newcommand{\KK}{\mathbb{K}}
\newcommand{\chr}[1]{\mathrm{char}\,#1}

\DeclareMathOperator{\Hom}{\mathrm{Hom}}
\DeclareMathOperator{\End}{\mathrm{End}}

\DeclareMathOperator{\Aut}{\mathrm{Aut}}

\DeclareMathOperator{\Stab}{\mathrm{Stab}}

\DeclareMathOperator{\AAut}{\mathbf{Aut}}
\DeclareMathOperator{\Der}{\mathrm{Der}}
\DeclareMathOperator{\supp}{\mathrm{Supp}}

\DeclareMathOperator{\Cent}{\mathrm{Cent}}

\DeclareMathOperator{\orbit}{\mathrm{orbit}}
\DeclareMathOperator{\Sym}{\mathrm{Sym}}

\newcommand{\Lie}{\mathrm{Lie}}

\newcommand{\frso}{{\mathfrak{so}}}

\newcommand{\GL}{\mathrm{GL}}

\newcommand{\Ort}{\mathrm{O}}

\newcommand{\Spin}{\mathrm{Spin}}

\newcommand{\Aff}{\mathrm{Aff}}

\newcommand{\Cs}{\mathbf{Cent}}
 
\newcommand{\Gs}{\mathbf{G}}
\newcommand{\Gsm}{\Gs_{\textup{m}}}

\newcommand{\Us}{\mathbf{U}}
\newcommand{\GLs}{\mathbf{GL}}

\newcommand{\Os}{\mathbf{O}}

\newcommand{\Spins}{\mathbf{Spin}}

\newcommand{\Diags}{\mathbf{Diag}}
\newcommand{\bmu}{\boldsymbol{\mu}}
\newcommand{\subdiag}{_{\textup{diag}}}
\newcommand{\subcartan}{_{\textup{Cartan}}}
\newcommand{\subCD}{_{\textup{CD}}}

\newenvironment{romanenumerate} 
{\begin{enumerate}

}{\end{enumerate}}

\newcommand{\subo}{_{\bar 0}} 
\newcommand{\subuno}{_{\bar 1}}

\newcommand{\Cl}{\mathfrak{Cl}} 

\newcommand{\bup}{\textup{b}} 
\newcommand{\nup}{\textup{n}}
\newcommand{\bnup}{\bup_\nup}
\newcommand{\hup}{\textup{h}}

\newcommand{\dc}{\textup{d}}

\newtheorem{theorem}{Theorem}[section]
\newtheorem{proposition}[theorem]{Proposition}
\newtheorem{lemma}[theorem]{Lemma}
\newtheorem{corollary}[theorem]{Corollary}

\theoremstyle{definition} 
\newtheorem{df}[theorem]{Definition}
\newtheorem{example}[theorem]{Example}

\theoremstyle{remark} \newtheorem{remark}[theorem]{Remark}
\numberwithin{equation}{section}

\begin{document}

\title{Cross products, automorphisms, and gradings}

\author[A.~Daza-Garc{\'\i}a]{Alberto Daza-Garc{\'\i}a}
\address[A.\,D., A.\,E.]{Departamento de
Matem\'{a}ticas e Instituto Universitario de Matem\'aticas y
Aplicaciones, Universidad de Zaragoza, 50009 Zaragoza, Spain}
\email{albertodg@unizar.es, elduque@unizar.es} 
\thanks{The first two authors are supported by grant
MTM2017-83506-C2-1-P (AEI/FEDER, UE) and by grant E22\_17R (Gobierno de 
Arag\'on, Grupo de referencia ``\'Algebra
y Geometr{\'\i}a''), cofunded by Feder 2014-2020 ``Construyendo
Europa desde Arag\'on''.\\
The first author also acknowledges support by the F.P.I. 
grant PRE2018-087018}

\author[A.~Elduque]{Alberto Elduque} 

\author[L.~Tang]{Liming Tang}
\address[L.\,T.]{School of Mathematical
Sciences, Harbin Normal
University, 150025 Harbin, China} 
\email{limingtang@hrbnu.edu.cn} 
\thanks{The third author is supported by grants CSC, 11971134
 (NSF of China), A2017005 (NSF of HLJ Province) and  XKB2014003 
(NSF of HNU).\\
She would like to thank the ``Departamento de Matem\'aticas''  and the ``Instituto Universitario de Matem\'aticas y
Aplicaciones'' of the University of Zaragoza for their support and hospitality during her visit from October 2019 to July 2020.}




\begin{abstract}
The affine group schemes of automorphisms of the multilinear $r$-fold 
cross products  on finite-dimensional vectors spaces over fields of 
characteristic not two are determined. Gradings by abelian groups 
on these structures, that correspond to morphisms from diagonalizable 
group schemes into these group schemes of automorphisms, are 
completely classified, up to isomorphism.
\end{abstract}

\maketitle

\section{Introduction}\label{se:intro}

Eckmann \cite{Eckmann} defined a \emph{vector cross product} on
an $n$-dimensional real vector space $V$, endowed with a
(positive definite) inner product $\bup(u,v)$, to be a continuous
map 
\[ 
X:V^r\longrightarrow V\qquad (1\leq r\leq n) 
\]
satisfying the following axioms: 
\begin{align}
&\bup\bigl(X(v_1,\ldots,v_r),v_i\bigr)=0,\quad 1\leq i\leq
r,\label{eq:axioms1}\\
&\bup\bigl(X(v_1,\ldots,v_r),X(v_1,\ldots,v_r)\bigr)=\det\Bigl(\bup(v_i,v_j)\Bigr),\label{eq:axioms2}
\end{align}

There are very few possibilities.

\begin{theorem}[\cite{Eckmann,Whitehead}]\label{th:EW} 
A vector
cross product exists in precisely the following cases:
\begin{itemize} 
\item $n$ is even, $r=1$, 
\item $n\geq 3$,
$r=n-1$, \item $n=7$, $r=2$, 
\item $n=8$, $r=3$. 
\end{itemize}
\end{theorem}

Multilinear vector cross products $X$ on vector spaces $V$ over
arbitrary fields of characteristic not two, relative to a
nondegenerate symmetric bilinear form $\bup(u,v)$, were
classified by Brown and Gray \cite{BG}. These are the
multilinear maps $X:V^r\rightarrow V$ ($1\leq r\leq n$)
satisfying \eqref{eq:axioms1} and \eqref{eq:axioms2}.
The possible pairs $(n,r)$ are again those in Theorem \ref{th:EW}.

The exceptional cases: $(n,r)=(7,2)$ and $(8,3)$, are
intimately related to the octonion, or Cayley, algebras.

These multilinear vector cross products have become important
tools in Differential Geometry or Nonassociative Algebras (see, for
instance, \cite{Karigiannis,EKO}). In particular, they are closely related
to the exceptional classical simple Lie superalgebras: $G(3)$,
$F(4)$, $D(2,1;\alpha)$ (\cite{KO}).

An elementary account on multilinear cross products, and of their 
connections with the exceptional basic classical simple Lie 
superalgebras can be found in \cite{Eld_Talk}.

\medskip

The nondegenerate symmetric bilinear form $\bup(u,v)$ is part of
the definition given by Eckmann and Brown and Gray. However,
this is not, in general, uniquely determined. Therefore, we will
make use of the following definition.

\begin{df}\label{df:CP} 
Let $V$ be a finite-dimensional vector
space over a field $\FF$ of characteristic not two, and let $r$
be a natural number with $1\leq r\leq n$ ($n=\dim_\FF V$).

An \emph{$r$-fold cross product} $X$ on $V$ is a multilinear map
\[ 
X:V^r\longrightarrow V 
\] 
such that there is a nondegenerate
symmetric bilinear form $\bup:V\times V\rightarrow \FF$ satisfying
 conditions \eqref{eq:axioms1} and \eqref{eq:axioms2}.
\end{df}

In this situation we will say that $X$ is a cross product on $V$
\emph{relative to $\bup$}, or that $\bup$ \emph{admits} the 
$r$-fold cross product $X$.

Alternatively, we will simply say that $(V,X)$ (or $(V,X,\bup)$,
if $\bup$ is fixed) is a cross product.

Two such pairs $(V_1,X_1)$ and $(V_2,X_2)$ are said to be
isomorphic if the rank $r$ is the same in both cases and there
is a linear isomorphism $\varphi:V_1\rightarrow V_2$ such that
$\varphi\bigl(X_1(v_1,\ldots,v_r)\bigr)=
X_2\bigl(\varphi(v_1),\ldots,\varphi(v_r)\bigr)$ for any
$v_1,\ldots,v_r\in V_1$; while two such triples
$(V_1,X_1,\bup_1)$ and $(V_2,X_2,\bup_2)$ are said to be
isomorphic if there is a linear isometry
$\varphi:(V_1,\bup_1)\rightarrow (V_2,\bup_2)$ such that
$\varphi\bigl(X(v_1,\ldots,v_r)\bigr)
=X\bigl(\varphi(v_1),\ldots,\varphi(v_r)\bigr)$ for any
$v_1,\ldots,v_r\in V_1$.

\smallskip

Therefore, given a cross product $X$ on the vector space $V$
relative to the bilinear form $\bup$, we will consider two 
\emph{automorphism groups}, depending on whether the bilinear form
$\bup$ is considered a part of the definition: 
\[ 
\begin{split}
\Aut(V,X)&\bydef\{\varphi\in \GL(V)\mid\\[-1pt]
&\null\qquad\qquad\varphi\bigl(X(v_1,\ldots,v_r)\bigr)=
X\bigl(\varphi(v_1),\ldots,\varphi(v_r)\bigr)\ 
\forall v_1,\ldots,v_r\in V\},\\ 
\Aut(V,X,\bup)&\bydef \Aut(V,X)\cap \Ort(V,\bup), 
\end{split} 
\] 
where $\GL(V)$ denotes the general linear group of $V$, and 
$\Ort(V,\bup)$ the orthogonal group of $(V,\bup)$. The last group
$\Aut(V,X,\bup)$ is the one considered in \cite{BG}.

\smallskip

More generally, we will consider the corresponding affine group
schemes, which will be treated in a functorial point of view
(see \cite{Waterhouse} or \cite[Chapter VIII]{KMRT}). That is,
for any unital, associative, commutative $\FF$-algebra $R$, the
corresponding group of $R$-points are the following: 
\[
\begin{split} 
\AAut(V,X)(R)&=\{\varphi\in \GLs(V)(R)\mid \\[-1pt]
&\null\qquad\varphi\bigl(X_R(v_1,\ldots,v_r)\bigr)=
X_R\bigl(\varphi(v_1),\ldots,\varphi(v_r)\bigr)\ 
 \forall v_1,\ldots,v_r\in V_R\},\\[4pt]
\AAut(V,X,\bup)(R)&= \AAut(V,X)(R)\cap \Os(V,\bup)(R),
\end{split} 
\] 
where $V_R=V\otimes_\FF R$, $X_R$ denotes the
scalar extension of $X$ to $V_R$, and $\GLs(V)$ and
$\Os(V,\bup)$ denote the general linear affine group scheme and
the orthogonal group scheme attached to $V$ and $\bup$.

\medskip

\emph{All vector spaces considered from now on will be assumed
to be finite-dimensional and defined over a ground field $\FF$
of characteristic not two.}

\medskip

The paper is organized as follows. Section \ref{se:cross} will be
 devoted to reviewing the basic results and examples of cross products
 and will answer the question as to what extent the form $\bup$ is
 uniquely determined by the cross product. In section \ref{se:autos},
the affine group schemes of automorphisms of $r$-fold cross products 
will be determined. In the most interesting cases, these are simply 
connected algebraic groups of types $G_2$ and $B_3$. In case 
$r=\dim_\FF V-1$, these group schemes are close to the special 
orthogonal group scheme, but they may fail to be smooth. 
Finally, Section 
\ref{se:gradings} will be devoted to classifying gradings by abelian 
groups, up to isomorphism, on cross products. Fine gradings, up to 
equivalence, will be classified too, and the associated Weyl groups will 
be computed. In future work, these classifications will be used to 
study gradings by abelian groups on the exceptional simple 
basic classical Lie superalgebras.

\bigskip


\section{Cross products}\label{se:cross}

In this section, some results on cross products will be reviewed
in a way suitable for our purposes. The problem of the uniqueness of
the associated bilinear form will be tackled too.

\smallskip

If $(V,X,\bup)$ is an $r$-fold cross product, then it is clear
that for any nonzero scalars $\alpha,\beta\in\FF$, with
$\beta^{r-1}=\alpha^2$,  $(V,\alpha X,\beta\bup)$ is an
$r$-fold cross product too.

Let us review the main examples of cross products. First, let us
recall from \cite[\S 3]{BG} the notion of star operator. Let
$\bup$ be a nondegenerate symmetric bilinear form of
discriminant $1$ on the vector space $V$ of dimension $n$. Then
$\bup$ extends to the exterior algebra $\bigwedge V$ as follows:
\[ 
\bup\bigl(u_1\wedge\cdots\wedge u_p,v_1\wedge\cdots\wedge
v_q\bigr)
=\begin{cases}
\det\Bigl(\bup(u_i,v_j)\Bigr)\quad\text{if $p=q$,}\\
0\quad\text{otherwise.}
\end{cases} 
\] 
As the discriminant is
$1$, there exists an element $\omega\in\bigwedge^nV$, unique up
to sign, such that $\bup(\omega,\omega)=1$. The \emph{star
operator} relative to $\bup$ and $\omega$ is the linear map
${}^*:\bigwedge V\rightarrow \bigwedge V$, such that
${}^*\bigl(\bigwedge^pV\bigr)=\bigwedge^{n-p}V$, $0\leq p\leq
n$, defined by 
\[ 
\bup\bigl({}^*x,y\bigr)=\bup\bigl(x\wedge y,\omega\bigr), 
\] 
for any $0\leq p\leq n$, $x\in\bigwedge^pV$,
$y\in\bigwedge^{n-p}V$.

In this case, the multilinear map 
\[ 
\begin{split}
X:V^{n-1}&\longrightarrow V\\
(v_1,\ldots,v_{n-1})&\mapsto{}^*(v_1\wedge\cdots\wedge v_{n-1})
\end{split} 
\]
is an $(n-1)$-fold cross product on $V$ relative to $\bup$.

\smallskip

Let now $\cC$ be a Cayley algebra with norm $\nup$. That is,
$\cC$ is an eight-dimensional unital nonassociative algebra
endowed with a nondegenerate quadratic form $\nup$ admitting
composition: 
\[ 
\nup(xy)=\nup(x)\nup(y) 
\] 
for all $x,y\in\cC$.
Any $x\in\cC$ satisfies the Cayley-Hamilton equation: 
\begin{equation}\label{eq:CH}
x^2-\nup(x,1)x+\nup(x)1=0 
\end{equation}
where
$\nup(x,y)=\nup(x+y)-\nup(x)-\nup(y)$ is the polar form of
$\nup$. Write $\bnup(x,y)=\frac{1}{2}\nup(x,y)$. The linear map
$x\mapsto \bar x\bydef\nup(x,1)1-x$ is the canonical involution
of $\cC$. It satisfies the equations
\begin{equation}\label{eq:nbar}
\nup(xy,z)=\nup(y,\bar xz)=\nup(x,z\bar y)
\end{equation}
for all $x,y,z\in\cC$.
Any two elements in $\cC$ generate an associative subalgebra.

Let $\cC_0$ be the subspace orthogonal to the unity $1$:
$\cC_0=\{x\in\cC\mid \nup(x,1)=0\}$. For any $x,y\in\cC_0$, the
product $xy$ in $\cC$ splits as: 
\begin{equation}\label{eq:C0xy}
xy=-\bnup(x,y)1+x\times y 
\end{equation} 
for some $x\times
y\in\cC_0$. Then $\times$ is an anticommutative multiplication
on $\cC_0$ that satisfies: 
\begin{equation}\label{eq:C0xyy}
(x\times y)\times y=\bnup(x,y)y-\bnup(y,y)x 
\end{equation} 
for all $x,y\in\cC_0$ (see, e.g., \cite[Chapter VIII]{KMRT} or
\cite[Chapter 4]{EKmon}).

Moreover, $(\cC_0,X^{\cC_0},\bnup)$ is a $2$-fold cross product,
with $X^{\cC_0}(x,y)=x\times y$.

\smallskip

On the other hand, $(\cC,X^\cC_\epsilon,\bnup)$, with
$\epsilon=\pm 1$, where $X^\cC_\epsilon$ is given by the
formulas:  
\begin{align} 
X^\cC_1(x,y,z)&=(x\bar y)z-\bnup(x,y)z-\bnup(y,z)x+\bnup(x,z)y,
\label{eq:typeI}\\
X^\cC_{-1}(x,y,z)&=x(\bar yz)-\bnup(x,y)z-\bnup(y,z)x+\bnup(x,z)y,
\label{eq:typeII} 
\end{align} 
is a $3$-fold cross product on $\cC$ (see \cite[Theorem 5.1]{BG}).

Moreover, the following holds for all $u_i,v_i\in\cC$, $1\leq
i\leq 3$ (see \cite[Proposition 3]{Korean}):
\begin{multline}\label{eq:Ctype}
\bnup\Bigl(X^\cC_\epsilon(u_1,u_2,u_3),X^\cC_\epsilon(v_1,v_2,v_3)\Bigr)
=\det\Bigl(\bnup(u_i,v_j)\Bigr)\\ + \epsilon\sum_{\sigma\
\text{even}}\sum_{\tau\ \text{even}}
\bnup(u_{\sigma(1)},v_{\tau(1)}) \bnup\Bigl(u_{\sigma(2)},
X^\cC_\epsilon(u_{\sigma(3)},v_{\tau(2)},v_{\tau(3)})\Bigr).
\end{multline} 
($\epsilon$ is either $1$ or $-1$, and the sums
are over the even permutations of $1,2,3$.)

\medskip

Now we have all the ingredients to review the classification of
cross products. The trace of a linear
operator $X$ will be denoted by $\tr(X)$.

\begin{theorem}\label{th:CP} 
Let $X$ be an $r$-fold cross
product on a vector space $V$ of dimension $n$, $1\leq r\leq n$.
Then one, and only one, of the following conditions holds:
\begin{romanenumerate} 
\item $n$ is even, $r=1$, $X^2=-\id$, and
$\tr(X)=0$.

\item $n\geq 3$, $r=n-1$, and 
\[
X(v_1,\ldots,v_{n-1})={}^*(v_1\wedge\cdots\wedge v_{n-1}) 
\] 
for all $v_1,\ldots,v_{n-1}\in V$, where $*$ is the star operator
relative to a nondegenerate symmetric bilinear form $\bup$ of
discriminant $1$ and an element $\omega\in\bigwedge^nV$ with
$\bup(\omega,\omega)=1$.

\item $n=7$, $r=2$, and $(V,X)$ is isomorphic to
$(\cC_0,X^{\cC_0})$ for a Cayley algebra $\cC$.

\item $n=8$, $r=3$, and $(V,X)$ is isomorphic to $(\cC,\alpha
X^{\cC}_1)$ for a Cayley algebra $\cC$ and a nonzero scalar
$\alpha\in \FF$. 
\end{romanenumerate}

Conversely, all the pairs $(V,X)$ in items \textup{(i)--(iv)} are cross
products. 
\end{theorem}

\begin{proof} 
With the exception of a few details, everything
follows from \cite{BG}. In particular, it is proved in \cite{BG}
that the only possible pairs $(n,r)$ are $(2s,1)$, $(n,n-1)$,
$(7,2)$ and $(8,3)$.

If $r=1$ and $X$ is a $1$-fold cross product relative to the
bilinear form $\bup$, then $X$ satisfies
$\bup\bigl(X(u),u\bigr)=0$ and
$\bup\bigl(X(u),X(v)\bigr)=\bup(u,v)$ for all $u,v\in V$. That
is, $X$ is both a skew-symmetric transformation and an isometry
relative to $\bup$. If $X^*$ denotes the adjoint of the
endomorphism $X$ relative to $\bup$, then we have $X^*=-X$ and
$XX^*=\id$, and this implies $X^2=-\id$ and $\tr(X)=0$.

Conversely, let $X:V\rightarrow V$ be an endomorphism such that
$X^2=-\id$ and $\tr(X)=0$. If $-1\in\FF^2$ and we pick
$\bi\in\FF$ with $\bi^2=-1$, then $V=V_+\oplus V_-$ with
$V_{\pm}=\{v\in V\mid X(v)=\pm\bi v\}$. The condition $\tr(X)=0$
gives $\dim_\FF V_+=\dim_\FF V_-$ so $n$ is even and $X$ is a
$1$-fold cross product relative to any nondegenerate symmetric
bilinear form $\bup:V\times V\rightarrow \FF$ with
$\bup(V_+,V_+)=0=\bup(V_-,V_-)$.

If $-1\not\in\FF^2$, then the subalgebra $\KK=\FF\id\oplus\FF X$
of $\End_\FF(V)$ is a field, and $V$ is a vector space over
$\KK$. For any $\KK$-basis $\{v_1,\ldots,v_s\}$ of $V$, $n=2s$
and $X$ is a $1$-fold cross product relative to the bilinear
form $\bup$ defined so that the $\FF$-basis
$\{v_1,X(v_1),\ldots,v_s,X(v_s)\}$ is orthogonal and
$\bup(v_i,v_i)=\bup\bigl(X(v_i),X(v_i)\bigr)=1$ for all $i$.

The case $n\geq 3$, $r=n-1$ is proved in \cite[Theorem 3.3]{BG}.
For $n=7$, $r=2$ it follows from \cite[Theorem 4.1]{BG}.

Finally, if $n=8$, $r=3$, and $(V,X,\bup)$ is a $3$-fold cross
product, then we have \eqref{eq:Ctype} for $\epsilon=1$ (type
I) or $\epsilon=-1$ (type II). If $(V,X,\bup)$ is of type I,
then $(V,X,-\bup)$ is of type II. Hence, changing $\bup$ by
$-\bup$ if necessary, we may assume that $(V,X,\bup)$ is of type
I and (iv) follows from \cite[Theorem 5.1]{BG} and
\cite[Proposition 3]{Korean}. 
\end{proof}

\begin{remark}\label{re:uniqueness_b} 
The proof of Theorem
\ref{th:CP} shows that the bilinear form $\bup$ is not
determined at all from the cross product $X$ for $r=1$. Equation
\eqref{eq:C0xy} shows, on the other hand, that $\bup$ is uniquely
determined by $X$ in case $n=7$, $r=2$. 
\end{remark}

The remaining cases are dealt with in the next result.

\begin{proposition}\label{pr:uniqueness_b} 
Let $X$ be an
$r$-fold cross product on an $n$-dimensional vector space $V$
relative to the nondegenerate symmetric bilinear forms $\bup$
and $\bup'$. 
\begin{itemize} 
\item 
If $n\geq 3$ and $r=n-1$,
then there is a scalar $\mu\in\FF$ with $\mu^{n-2}=1$, such that
$\bup'=\mu\bup$. \\
In particular, for $n=3$, $\bup$ is uniquely determined.
\item 
If $n=8$ and $r=3$, then $\bup'$ equals
either $\bup$ or $-\bup$. In particular, the bilinear form
$\bup$ is unique if $(V,X,\bup)$ is assumed to be of type I.
\end{itemize} 
\end{proposition}

\begin{proof} 
Let $n\geq 3$, $r=n-1$, and without loss of
generality, assume that $\FF$ is algebraically closed. Let
$\{v_1,\ldots,v_n\}$ be an orthogonal basis relative to $\bup$,
with $\bup(v_i,v_i)=1$ for all $i$.

As $\bup\bigl(X(v_1,\ldots,v_{n-1}),v_i\bigr)=0$ for all
$i=1,\ldots,n-1$, and
\[
\bup\bigl(X(v_1,\ldots,v_{n-1}),X(v_1,\ldots,v_{n-1})\bigr)=\det\Bigl(\bup(v_i,v_j)\Bigr)=1,
\]
changing if necessary $v_n$ by $-v_n$, we may assume that
$X(v_1,\ldots,v_{n-1})=v_n$. From the fact that
$\Phi:V^n\rightarrow \FF$, given by
$\Phi(u_1,\ldots,u_n)=\bup\bigl(X(u_1,\ldots,u_{n-1}),u_n\bigr)$,
is multilinear and alternating, we conclude that for any
permutation $\sigma$, 
\begin{equation}\label{eq:Xsigma}
X(v_{\sigma(1)},\ldots,v_{\sigma(n-1)})=(-1)^\sigma
v_{\sigma(n)}, 
\end{equation}
where $(-1)^\sigma$ denotes the signature of $\sigma$.

But $\Phi':V^n\rightarrow \FF$ given by
$\Phi'(u_1,\ldots,u_n)=\bup'\bigl(X(u_1,\ldots,u_{n-1}),u_n\bigr)$,
is also alternating, and hence there is a nonzero scalar $\mu$
such that $\Phi'=\mu\Phi$. For any permutation $\sigma$, 
\[
\bup'(v_{\sigma(n)},v_{\sigma(i)})=
(-1)^\sigma\Phi'\bigl(v_{\sigma(1)},\ldots,v_{\sigma(n-1)},v_{\sigma(i)}\bigr) 
=\mu\bup(v_{\sigma(n)},v_{\sigma(i)}),
\] 
so that we get $\bup'=\mu\bup$. Equation \eqref{eq:axioms2}
gives $\mu^{n-2}=1$.

\smallskip

In case $n=8$, $r=3$, assume again that $\FF$ is algebraically
closed and let $\Phi:V^4\rightarrow \FF$ (respectively
$\Phi':V^4\rightarrow \FF$) be the alternating multilinear map
given by
$\Phi(u_1,u_2,u_3,u_4)=\bup\bigl(X(u_1,u_2,u_3),u_4\bigr)$
(respectively
$\Phi'(u_1,u_2,u_3,u_4)=\bup'\bigl(X(u_1,u_2,u_3),u_4\bigr)\,$).
Using \eqref{eq:axioms2} we get: 
\[ 
\begin{split}
\begin{vmatrix}
\bup(x,x)&\bup(x,y)&\bup(x,v)\\ 
\bup(y,x)&\bup(y,y)&\bup(y,v)\\
\bup(u,x)&\bup(u,y)&\bup(u,v) 
\end{vmatrix}
&=\bup\bigl(X(x,y,u),X(x,y,v)\bigr)\\[-6pt]
&=\Phi\bigl(x,y,u,X(x,y,v)\bigr)\\
&=-\Phi\bigl(x,y,X(x,y,v),u\bigr)\\
&=-\bup\bigl(X(x,y,X(x,y,v)),u\bigr) 
\end{split} 
\] 
The nondegeneracy of $\bup$ gives 
\[ 
X(x,y,X(x,y,v))=
\begin{vmatrix} \bup(x,x)&\bup(x,v)\\ \bup(y,x)&\bup(y,v)
\end{vmatrix}y 
- \begin{vmatrix} \bup(x,y)&\bup(x,v)\\
\bup(y,y)&\bup(y,v) \end{vmatrix}x 
- \begin{vmatrix}
\bup(x,x)&\bup(x,y)\\ \bup(y,x)&\bup(y,y) \end{vmatrix}v, 
\] 
and hence we get 
\[ 
X\Bigl(x,y,X\bigl(x,y,X(x,y,v)\bigr)\Bigr)= -
\begin{vmatrix} \bup(x,x)&\bup(x,y)\\ \bup(y,x)&\bup(y,y)
\end{vmatrix}X(x,y,v), 
\] 
that is, 
\[ 
X(x,y,.)^3=-
\begin{vmatrix} \bup(x,x)&\bup(x,y)\\ \bup(y,x)&\bup(y,y)
\end{vmatrix}X(x,y,.) ,
\] 
and the same happens replacing $\bup$ by $\bup'$.

Thus, for any $x,y\in V$ with $X(x,y,.)\neq 0$ we obtain:
\begin{equation}\label{eq:bb'xy}
\bup(x,x)\bup(y,y)-\bup(x,y)^2=\bup'(x,x)\bup'(y,y)-\bup'(x,y)^2.
\end{equation} 
As the set of such pairs $(x,y)$ is Zariski dense
in $V\times V$, we conclude that \eqref{eq:bb'xy} holds for any
$x,y\in V$.

For any $x\in V$ with $\bup(x,x)\neq 0$, the subspace $S_x=\{
v\in V\mid \bup(x,v)=0=\bup'(x,v)\}$ has dimension at least $6$,
so it is not isotropic for both $\bup$ and $\bup'$. For any
$v\in S_x$, \eqref{eq:bb'xy} gives 
\[
\bup(v,v)=\frac{\bup'(x,x)}{\bup(x,x)}\bup'(v,v) 
\] 
so that
$\bup=\mu\bup'$ on the subspace $S_x$ with
$\mu=\frac{\bup'(x,x)}{\bup(x,x)}$. But \eqref{eq:bb'xy} gives 
\[
(1-\mu^2)\begin{vmatrix} \bup(u,u)&\bup(u,v)\\ 
    \bup(v,u)&\bup(v,v)\end{vmatrix}=0
\]
for any $u,v\in S_x$ and this implies $\mu^2=1$. Thus $\mu=\pm 1$ and
 hence $\bup'(x,x)=\pm \bup(x,x)$ for any $x\in V$ with 
$\bup(x,x)\neq 0$. Therefore, the polynomial map 
\[
\Bigl(\bup(x,x)-\bup'(x,x)\Bigr)\Bigl(\bup(x,x)+\bup'(x,x)\Bigr)
\]
 is $0$ on the Zariski dense set of nonisotropic vectors for $\bup$, so it is $0$, and hence so is one of its factors. Thus either $\bup'=\bup$ or $\bup'=-\bup$.
\end{proof}

\bigskip


\section{Automorphisms}\label{se:autos}

This section is devoted to computing the affine group schemes of
automorphisms of cross products.

\smallskip

\subsection{$n$ even, $r=1$}\quad

Let $X:V\rightarrow V$ be a $1$-fold cross product on the 
even-dimensional vector space $V$, relative to a 
nondegenerate symmetric bilinear form $\bup$. Theorem \ref{th:CP} 
and its proof tell us that 
$\KK=\FF\id\oplus \FF X$ is an \'etale $\FF$-algebra (either isomorphic
to $\FF\times \FF$ if $-1\in\FF^2$, or a quadratic field of 
$\FF$ otherwise). In any case, $\KK$ is endowed with a canonical involution: $X\mapsto -X$. Recall that $\bup$ is far from being determined by $X$.

Denote by $\Cs_{\GLs(V)}(X)$  the group scheme of elements 
that commute with $X$ in the general linear group scheme. 
The first part of the next result is trivial.

\begin{theorem}\label{th:2s1}
Let $X:V\rightarrow V$ be a $1$-fold cross product on the 
even-dimensional vector space $V$, relative to a 
nondegenerate symmetric bilinear form $\bup$.
\begin{itemize}
\item 
$\AAut(V,X)=\Cs_{\GLs(V)}(X)$.
\item
$\AAut(V,X,\bup)=\Us(V,\hup)$, where $\hup$ is the hermitian nondegenerate form given by
\[
\begin{split}
\hup: V\times V&\longrightarrow \KK\\
      (u,v)\, &\mapsto \bup(u,v)\id -\bup\bigl(X(u),v\bigr)X
\end{split}
\]
for any $u,v\in V$, and $\Us(V,\hup)$ is the corresponding \emph{unitary group scheme}, whose $R$-points are those
$\varphi\in\End_{\KK\otimes_\FF R}(V_R)$ such that 
$\hup_R\bigl(\varphi(u),\varphi(v)\bigr)=\hup_R(u,v)$ for all $u,v\in V_R=V\otimes_\FF R$. (The subindex $R$ denotes the natural scalar extension.)
\end{itemize}
\end{theorem}

\begin{proof}
The fact that $\hup$ is hermitian (i.e., $\hup$ is $\FF$-bilinear, 
$\hup\bigl(X(u),v\bigr)=X\hup(u,v)$, and $\hup(u,v)=\overline{\hup(v,u)}$, for all $u,v\in V$, where 
$\overline{\phantom{a}}$ denotes the canonical involution of $\KK$) is clear.

Note that for $\varphi\in\End_{\KK\otimes_\FF R}(V_R)$ and 
$u,v\in V_R$, $\hup_R\bigl(\varphi(u),\varphi(v)\bigr)=\hup_R(u,v)$ 
if, and only if,
$\bup_R\bigl(\varphi(u),\varphi(v)\bigr)=\bup_R(u,v)$ and 
$\bup_R\bigl(X_R(\varphi(u)),\varphi(v)\bigr)
=\bup_R\bigl(X_R(u),v\bigr)$. 
Hence
$\varphi$ is an element of $\Us(V,\hup)(R)$ if and only if $\varphi$ is an $R$-point in
 the intersection of $\Cs_{\GLs(V)}(X)=\AAut(V,X)$ and of the orthogonal group scheme $\Os(V,\bup)$, whence the result.
\end{proof}

In case $-1\in\FF^2$, $\KK$ is split and Theorem \ref{th:2s1} gives:

\begin{corollary}\label{co:2s1}
Let $X:V\rightarrow V$ be a $1$-fold cross product on the 
even-dimensional vector space $V$, relative to a 
nondegenerate symmetric bilinear form $\bup$, over a field $\FF$ containing a square root $\bi$ of $-1$. Then, the following conditions hold:
\begin{itemize}
\item 
$\AAut(V,X)$ is isomorphic to $\GLs(V_+)\times \GLs(V_-)$, where $V_{\pm}=\{ v\in V\mid X(v)=\pm\bi v\}$.
\item
$\AAut(V,X,\bup)$ is isomorphic to $\GLs(V_+)$.
\end{itemize}
\end{corollary}
\begin{proof}
Any automorphism of $(V,X)$ preserves $V_+$ and  $V_-$, and any automorphism of $(V,X,\bup)$ is determined by its action on $V_+$, because $V_+$ and $V_-$ are paired by $\bup$.
\end{proof}

\bigskip

\subsection{$n\geq 3$, $r=n-1$}\quad

Given a vector space $V$ endowed with a nondegenerate symmetric
bilinear form, we will consider the \emph{special orthogonal group
scheme} $\Os^+(V,\bup)$ and the group scheme $\widetilde\Os(V,\bup)$
whose $R$-points are those invertible linear automorphisms $\varphi$
of $V_R$ such that 
$\bup_R\bigl(\varphi(u),\varphi(v)\bigr)=\det(\varphi)\bup_R(u,v)$ for
all $u,v\in V_R$.

\begin{theorem}\label{th:n-1}
Let $X:V^{n-1}\rightarrow V$ be an ($n-1$)-fold cross product on the 
$n$-dimensional vector space $V$, relative to a 
nondegenerate symmetric bilinear form $\bup$.
Then the next two equations hold:
\begin{itemize}
\item 
$\AAut(V,X)=\widetilde\Os(V,\bup)$.
\item
$\AAut(V,X,\bup)=\Os^+(V,\bup)$.
\end{itemize}
\end{theorem}

\begin{proof}
Let $\varphi$ be a  linear automorphism of $V$. After extending scalars 
we may assume, as in the proof of Proposition \ref{pr:uniqueness_b}, 
that there exists a basis $\{v_1,\ldots,v_n\}$ satisfying Equation 
\eqref{eq:Xsigma}. Then $\varphi\in\Aut(V,X)$ if, and only if, we have
\[
X\bigl(\varphi(v_{\sigma(1)}),\ldots,\varphi(v_{\sigma(n-1)})\bigr)
=\varphi\bigl(X(v_{\sigma(1)},\dots,v_{\sigma(n-1)})\bigr)
=(-1)^\sigma\varphi(v_{\sigma(n)})
\]
for any permutation $\sigma$, and this happens, as $\bup$ is nondegenerate, if and only if we have 
\[
\begin{split}
\bup\bigl(\varphi(v_{\sigma(n-1)}),\varphi(v_{\sigma(i)})\bigr)
&=(-1)^\sigma\Phi\bigl(\varphi(v_{\sigma(1)}),\ldots,
\varphi(v_{\sigma(n-1)}),\varphi(v_{\sigma(i)})\bigr)\\
 &=(-1)^\sigma\det(\varphi)\Phi(v_{\sigma(1)},\ldots,
      v_{\sigma(n-1)},v_{\sigma(i)})\\
   &=\begin{cases}
       \det(\varphi)&\text{if $i=n$}\\ 0&\text{otherwise}
     \end{cases}\\
  &=\det(\varphi)\bup\bigl(v_{\sigma(n)},v_{\sigma(i)}\bigr)
\end{split}
\]
where $\Phi$ is the alternating multilinear form considered in the proof 
of Proposition \ref{pr:uniqueness_b}. We conclude that $\varphi$ is an automorphism of $(V,X)$ if and only if it satisfies 
\begin{equation}\label{eq:bphidet}
\bup\bigl(\varphi(u),\varphi(v)\bigr)=\det(\varphi)\bup(u,v)
\end{equation}
for all $u,v\in V$.

The argument above is functorial, so we conclude $\AAut(V,X)=
\widetilde\Os(V,\bup)$. Moreover, $\AAut(V,X,\bup)=\AAut(V,X)\cap
\Os(V,\bup)=\Os^+(V,\bup)$.
\end{proof}

Denote by $\bmu_m$ the scheme of $m$-th roots of unity.
A natural short exact sequence appears:

\begin{proposition}\label{pr:exact}
Let $X:V\rightarrow V$ be an ($n-1$)-fold cross product on the 
$n$-dimensional vector space $V$, relative to a 
nondegenerate symmetric bilinear form $\bup$. Then
the determinant provides a short exact sequence:
\[
1\longrightarrow \Os^+(V,\bup)\longrightarrow \widetilde\Os(V,\bup)
\stackrel{\det}
\longrightarrow \bmu_{n-2}\longrightarrow 1\,.
\]
\end{proposition}

\begin{proof}
For any (unital, associative, commutative) $\FF$-algebra $R$ and any 
element
$\varphi\in\widetilde\Os(V,\bup)(R)$, Equation \eqref{eq:bphidet} shows that $\det(\varphi)^2=\det(\varphi)^n$, so that  
$\det(\varphi)$ is an ($n-2$)-th root of unity, thus showing that 
the exact sequence is well defined. Now, the only thing left is to 
prove that $\det$ gives a quotient map 
$\widetilde\Os(V,\bup)\rightarrow \bmu_{n-2}$.

Given any $\FF$-algebra $R$ and a root of unity $r\in\bmu_{n-2}(R)$, consider the degree two extension $S=R[T]/(T^2-r)$. Denote by $t$ the class of $T$ modulo $(T^2-r)$. 
The algebra $S$ is a free $R$-module of rank two, so $S/R$ is a faithfully flat extension.
Take an orthogonal basis $\{v_1,\ldots,v_n\}$ and consider the linear
automorphism $\varphi\in\GLs(V)(S)$ such that
\[
\varphi(v_i)=tv_i,\ 1\leq i\leq n;\quad \varphi(v_n)=t^{n-1}v_n.
\]
Note that $t^2=r=r^{n-1}=(t^{n-1})^2$, so that $\varphi$ lies in 
$\widetilde\Os(V,\bup)(S)$, and $\det(\varphi)=(t^{n-1})^2=r$.

This shows that $\det:\widetilde\Os(V,\bup)\rightarrow \bmu_{n-2}$
is a quotient map.
\end{proof}

In particular, for $n=3$, $\AAut(V,X)=\AAut(V,X,\bup)=\Os^+(V,\bup)$.

\begin{remark}\label{re:smoothOtilde}
As $\Os(V,\bup)$ is smooth, we get that $\widetilde\Os(V,\bup)$ is smooth if and only if so is $\bmu_{n-2}$ (see, e.g., \cite[(22.12)]{KMRT}). Therefore we obtain the following:
\begin{center}
\emph{
$\widetilde\Os(V,\bup)$ is smooth if and only if the characteristic of 
$\FF$ does not divide $n-2$.}
\end{center}
The Lie algebra of $\widetilde\Os(V,\bup)$ is 
\[
\Lie\bigl(\widetilde\Os(V,\bup)\bigr)=
\{ f\in\End_\FF(V)\mid \bup\bigl(f(u),v\bigr)+\bup\bigl(u,f(v)\bigr)=
\tr(f)\bup(u,v)\ \forall u,v\in V\}.
\]
Hence, for any $f\in \Lie\bigl(\widetilde\Os(V,\bup)\bigr)$,
$2\tr(f)=n\tr(f)$, so that $(n-2)\tr(f)=0$ and we get:
\[
\Lie\bigl(\widetilde\Os(V,\bup)\bigr)=\begin{cases}
\frso(V,\bup)&\text{if $\chr\FF$ does not divide $n-2$ (smooth 
case),}\\
\frso(V,\bup)\oplus\FF\id&\text{otherwise,}
\end{cases}
\]
where $\frso(V,\bup)$ denotes the orthogonal Lie algebra, i.e., the Lie algebra of skew-symmetric endomorphisms of $V$ relative to $\bup$.
\end{remark}

\bigskip

\subsection{$n=7$, $r=2$}\quad

If $(V,X)$ is a $2$-fold cross product on a seven-dimensional vector 
space, then $(V,X)$ is isomorphic to $(\cC_0,X^{\cC_0})$ for a Cayley 
algebra $\cC$ (Theorem \ref{th:CP}). Moreover (Remark
\ref{re:uniqueness_b}), the bilinear form is uniquely determined in 
this case. 

The restriction map
\[
\begin{split}
\AAut(\cC)&\longrightarrow \AAut(\cC_0,X^{\cC_0})\\
 f\ &\mapsto\quad f\vert_{\cC_0}
\end{split}
\]
is an isomorphism of affine group schemes (see, e.g., \cite{CastilloElduque}), and this gives the next result (see \cite{SpringerVeldkamp}).

\begin{theorem}\label{th:n7r2}
Let $(V,X,\bup)$ be a $2$-fold cross product on a seven-dimensional vector space. Then $\AAut(V,X)=\AAut(V,X,\bup)$ is a simple algebraic group of type $G_2$.
\end{theorem}

\bigskip

\subsection{$n=8$, $r=3$}\quad

If $(V,X)$ is a $3$-fold cross product on an eight-dimensional
vector space $V$, then by Theorem \ref{th:CP}, $(V,X)$ is isomorphic 
to $(\cC,\alpha X_1^\cC)$, for a Cayley algebra $\cC$ and a nonzero
scalar $\alpha\in\FF$, where $X_1^\cC$ is given by Equation 
\eqref{eq:typeI}. But $\AAut(V,X)=\AAut(V,\alpha X)$ for any nonzero 
$\alpha$, so it is enough to study the group scheme 
$\AAut(\cC,X_1^\cC)$. By Proposition \ref{pr:uniqueness_b}, this is
the same as $\AAut(\cC,X_1^\cC,\bnup)$.

Consider, as in \cite{Shaw1} or \cite{Korean} the triple product on 
the Cayley algebra $\cC$ given by
\begin{equation}\label{eq:3C}
\{xyz\}=(x\bar y)z
\end{equation}
that satisfies the equation
\begin{equation}\label{eq:3Cnorm}
\{xxy\}=\nup(x)y=\{yxx\}.
\end{equation} 
This is called a \emph{$3C$-product}, and the pair $(\cC,\{...\})$ a 
\emph{$3C$-algebra}.

Because of \eqref{eq:3Cnorm}, the norm of $\cC$ is determined by 
the $3C$-product, and hence we have the equality 
$\AAut(\cC,\{...\})=\AAut(\cC,\{...\},\nup)$. Moreover, Equation 
\eqref{eq:typeI} and Proposition \ref{pr:uniqueness_b} give:
\begin{equation}\label{eq:autos_pair83}
\AAut(\cC,X_1^\cC)=\AAut(\cC,X_1^\cC,\bnup)=
\AAut(\cC,\{...\})=\AAut(\cC,\{...\},\nup)
\end{equation}
and hence it is enough to compute $\AAut(\cC,\{...\})$. The group of 
rational points has been computed in \cite[Proposition 7]{Korean} and 
shown to be isomorphic to the spin group of $(\cC_0,-\nup)$. Here 
we will extend the results in \cite{Korean} to the group scheme 
setting, and will relate them to the triality phenomenon.

\smallskip

To begin with, let $\cC$ be a Cayley algebra with norm $\nup$, and
consider the \emph{para-Cayley} product 
$x\bullet y\bydef \bar x\bar y$. Let $l_x:y\mapsto x\bullet y$ and 
$r_x:y\mapsto y\bullet x$ be the operators of left and right multiplication by the element $x$ in the \emph{para-Cayley algebra}
$(\cC,\bullet)$. Consider, following \cite[\S 35]{KMRT} and 
\cite[\S 5.1]{EKmon}, the
linear map:
\[
\begin{split}
\cC&\longrightarrow \End_\FF(\cC\oplus\cC)\\
x\,&\mapsto\ \begin{pmatrix} 0&l_x\\ r_x&0\end{pmatrix}.
\end{split}
\]
For any $x,y\in \cC$, $r_xl_x(y)=(\overline{\bar x\bar y})\bar x=
(yx)\bar x=\nup(x)y=l_xr_x(y)$. Hence we have
\[
\begin{pmatrix} 0&l_x\\ r_x&0\end{pmatrix}^2=\nup(x)\id ,
\]
and the above linear map induces an isomorphism
\begin{equation}\label{eq:Phi}
\Phi:\Cl(\cC,\nup)\longrightarrow \End_\FF(\cC\oplus\cC) ,
\end{equation}
which restricts to an isomorphism of the even Clifford algebra 
$\Cl(\cC,\nup)\subo$ onto the diagonal subalgebra 
$\End_\FF(\cC)\times\End_\FF(\cC)$. Under this isomorphism, the 
\emph{canonical involution} $\tau$ of $\Cl(\cC,\nup)$: $\tau(x)=x$
for all $x\in\cC$, corresponds to the orthogonal involution 
$\sigma_{\nup\perp\nup}$ associated to the quadratic form 
$\nup\perp\nup$ on $\cC\oplus\cC$.

To prevent confusions between the unity of the Clifford algebra
$\Cl(\cC,\nup)$ and the unity of the Cayley algebra $\cC$, we will
denote in what follows the unity of $\Cl(\cC,\nup)$ by $1$, and the 
unity of $\cC$ by $e_0$. Also, the multiplication in $\cC$ will be denoted by juxtaposition, as usual, and the one in $\Cl(\cC,\nup)$ by
a dot: $x\cdot y$.

The associated spin group is the group:
\[
\Spin(\cC,\nup)=\{u\in\Cl(\cC,\nup)\subo\mid u\cdot \tau(u)=1,\ 
u\cdot \cC\cdot u^{-1}=\cC\},
\]
The \emph{vector representation} of $\Spin(\cC,\nup)$ on $\cC$ works as follows. For any $x\in\cC$ and $u\in\Spin(\cC,\nup)$, 
$\chi_u(x)\bydef u\cdot x\cdot u^{-1}=u\cdot x\cdot \tau(u)$.

For any $u\in\Spin(\cC,\nup)$, $\Phi(u)$ lies in the even part of 
$\End_\FF(\cC\oplus\cC)$:
\[
\Phi(u)=\begin{pmatrix}\rho_u^-&0\\ 0&\rho_u^{+}\end{pmatrix},
\]
and as $u\cdot\tau(u)=1$, it follows that $\rho_u^+$ and $\rho_u^-$ are orthogonal transformations: $\rho_u^\pm\in\Ort(\cC,\nup)$.

The equality $u\cdot x=\chi_u(x)\cdot u$ transfers by $\Phi$ to
\[
\begin{pmatrix}\rho_u^-&0\\ 0&\rho_u^{+}\end{pmatrix}
\begin{pmatrix}0&l_x\\ r_x&0\end{pmatrix}=
\begin{pmatrix}0&l_{\chi_u(x)}\\ r_{\chi_u(x)}&0\end{pmatrix}
\begin{pmatrix}\rho_u^-&0\\ 0&\rho_u^{+}\end{pmatrix}.
\]
This gives,
\[
\rho_u^-(x\bullet y)=\chi_u(x)\bullet\rho_u^+(y),\quad
\rho_u^+(y\bullet x)=\rho_u^-(y)\bullet\chi_u(x),
\]
for all $u\in\Spin(\cC,\nup)$ and $x,y\in\cC$. It is easy to check 
(\cite[Lemma 5.4]{EKmon}) that there is a cyclic symmetry here:
if three isometries 
$f_i\in\Ort(\cC,\nup)$, $0\leq i\leq 2$ satisfy $f_0(x\bullet y)=f_1(x)\bullet f_2(y)$ for all $x,y\in\cC$, then also $f_1(x\bullet y)=f_2(x)\bullet f_0(y)$, and $f_2(x\bullet y)=f_0(x)\bullet f_1(y)$.

All the above arguments are functorial and give an isomorphism of
affine group schemes (see \cite[\S 35]{KMRT})
\begin{equation}\label{eq:Spin_related_triples}
\begin{split}
\Spins(\cC,\nup)&\cong
\{(f_0,f_1,f_2)\in\Os^+(\cC,\nup)^3\mid f_0(x\bullet y)=f_1(x)\bullet f_2(y)\ \forall x,y\in\cC\}\\
u\ &\mapsto \ (\chi_u,\rho_u^+,\rho_u^-).
\end{split}
\end{equation}

\smallskip

Recall that we denote by $e_0$ the unity of the Cayley algebra $\cC$.

\begin{proposition}\label{pr:SpinC0}
The above isomorphism of affine group schemes induces an 
isomorphism
\begin{multline*}
\Spins(\cC_0,-\nup)\cong\{(f_0,f_1,f_2)\in\Os^+(\cC,\nup)^3\mid\\
 f_0(x\bullet y)=f_1(x)\bullet f_2(y)\ \text{and $f_0(e_0)=e_0$}\
 \forall x,y\in\cC\}.
\end{multline*}
\end{proposition}

\begin{proof}
 Consider the linear map
\[
\begin{split}
\cC_0&\longrightarrow \Cl(\cC,\nup)\\
 x\ &\mapsto\quad e_0\cdot x ,
\end{split}
\]
whose image lies in the even part $\Cl(\cC,\nup)\subo$.
Since $(e_0\cdot x)^{\cdot 2}=e_0\cdot x\cdot e_0\cdot x
=-e_0^{\cdot 2}\cdot x^{\cdot 2}=-\nup(e_0)\nup(x)1=-\nup(x)1$, because $e_0$ and $x$ are orthogonal relative to $\nup$, 
this linear map induces an embedding
\[
\Psi:\Cl(\cC_0,-\nup)\longrightarrow \Cl(\cC,\nup),
\]
which, by dimension count, restricts to an isomorphism (also denoted
by $\Psi$):
\begin{equation}\label{eq:Psi}
\Psi:\Cl(\cC_0,-\nup)\longrightarrow \Cl(\cC,\nup)\subo .
\end{equation}
Recall that $\tau$ denotes the canonical involution on $\Cl(\cC,\nup)$.
 Denote by $\tau'$ the involution on $\Cl(\cC_0,-\nup)$ which is 
$-\id$ on $\cC_0$. That is, $\tau'$ is the composition of the canonical
 involution on $\Cl(\cC_0,-\nup)$ with the parity automorphism. The restriction of $\tau'$ to the even part equals the restriction of
 the canonical involution.

For all $x\in\cC_0$ we get
\[
\Psi\bigl(\tau'(x)\bigr)=-\Psi(x)=-e_0\cdot x=x\cdot e_0=
\tau(x)\cdot\tau(e_0)=\tau(e_0\cdot x)=\tau\bigl(\Psi(x)\bigr),
\]
so that we have
\begin{equation}\label{eq:Psi_tau}
\Psi\tau'=\tau\Psi.
\end{equation}
As the elements of $\cC_0$ anticommute with $e_0$ in $\Cl(\cC,\nup)$
we get
\[
\begin{split}
\Psi\bigl(\Cl(\cC_0,-\nup)\subo\bigr)
  &=\{a\in\Cl(\cC,\nup)\subo\mid e_0\cdot a=a\cdot e_0\},\\
\Psi\bigl(\Cl(\cC_0,-\nup)\subuno\bigr)
  &=\{a\in\Cl(\cC,\nup)\subo\mid e_0\cdot a=-a\cdot e_0\}.
\end{split}
\]

In particular, for any $u\in\Spin(\cC_0,-\nup)$, we have 
$\Psi(u)\cdot e_0=e_0\cdot \Psi(u)$. Hence, for any $x\in\cC_0$
we get
\[
\Psi(u\cdot x\cdot u^{-1})
=\Psi(u)\cdot e_0\cdot x\cdot \Psi(u)^{-1}
=e_0\cdot \Psi(u)\cdot x\cdot \Psi(u)^{-1},
\]
so that we obtain
\[
\Psi(u)\cdot x\cdot \Psi(u)^{-1}
=e_0\cdot \Psi(u\cdot x\cdot u^{-1})
=e_0\cdot e_0\cdot \chi_u(x)
=\chi_u(x)\in\cC_0,
\]
and we conclude that $\Psi(u)$ lies in $\Spin(\cC,\nup)$.

Conversely, for any $v\in\Spin(\cC,\nup)$ with $e_0\cdot v=v\cdot e_0$, $v$ is the image $v=\Psi(u)$ of some element $u\in\Cl(\cC_0,-\nup)\subo$. From \eqref{eq:Psi_tau} we get $u\cdot \tau'(u)=1$, and 
for any $x\in \cC_0$, 
\[
\Psi(u\cdot x\cdot u^{-1})=
 v\cdot e_0\cdot x\cdot v^{-1}=e_0\cdot \chi_v(x)
=\Psi\bigl(\chi_v(x)\bigr)\in\Psi(\cC_0),
\]
and we conclude that $u\cdot x\cdot u^{-1}$ lies in $\cC_0$, so that 
$u$ lies in $\Spin(\cC_0,-\nup)$.

Therefore $\Psi$ restricts to a group isomorphism
\[
\Spin(\cC_0,-\nup)\cong \Cent_{\Spin(\cC,\nup)}(e_0),
\]
from the spin group of $(\cC_0,-\nup)$ onto the centralizer in $\Spin(\cC,\nup)$ of $e_0$ under the vector representation.

But all these arguments are functorial, so actually $\Psi$ induces an
isomorphism of affine group schemes:
\[
\Spins(\cC_0,-\nup)\cong \{u\in\Spins(\cC,\nup)\mid u\cdot e_0=e_0\cdot u\}.
\]

Finally, if we compose this isomorphism with the one in Equation
\eqref{eq:Spin_related_triples} we obtain the isomorphism of affine group schemes:
\begin{equation}\label{eq:SpinC0_related}
\begin{split}
\Spins(\cC_0,-\nup)&\longrightarrow \{(f_0,f_1,f_2)\in\Os^+(\cC,\nup)^3
  \mid f_0(x\bullet y)=f_1(x)\bullet f_2(y)\\
 &\hspace*{2in} \text{and $f_0(e_0)=e_0$}
\ \forall x,y\in\cC\}\\
  u\ &\mapsto\quad 
\bigl(\chi_{\Psi(u)},\rho_{\Psi(u)}^+,\rho_{\Psi(u)}^-\bigr),
\end{split}
\end{equation}
as required.
\end{proof}

\begin{remark}\label{re:fe0}
If $\cC$ is a Cayley algebra, and $(f_0,f_1,f_2)$ is a triple of
 isometries satisfying $f_0(x\bullet y)=f_1(x)\bullet f_2(y)$ for all
 $x,y\in\cC$, then we also have $f_1(x\bullet y)=f_2(x)\bullet f_0(y)$.
If $f_0(e_0)=e_0$, then with $y=e_0$ above we get 
$f_1(\bar x)=\overline{f_2(x)}$ for all $x$, so that 
$f_1(x)=\overline{f_2(\bar x)}$. Then both $f_0$ and $f_1$ are determined by $f_2$.

Conversely, with $(f_0,f_1,f_2)$ as above, if $f_1(\bar x)=
\overline{f_2(x)}$ for all $x\in\cC$, we get
\[
f_1(\bar x)=f_1(x\bullet e_0)
=\begin{cases}
  f_2(x)\bullet f_0(e_0),\\
  \overline{f_2(x)}=f_2(x)\bullet e_0,
 \end{cases}
\]
and we conclude that $f_0(e_0)=e_0$. 

As both $f_0$ and $f_1$ are then determined by $f_2$,  projecting 
on the third component in \eqref{eq:SpinC0_related} gives an injective
homomorphism (i.e., a closed embedding):
\begin{equation}\label{eq:theta}
\begin{split}
\theta:\Spins(\cC_0,-\nup)&\longrightarrow \Os^+(\cC,\nup)\\
 u\ &\mapsto\ \rho_{\Psi(u)}^-.
\end{split}
\end{equation}
This is actually the spin representation of $\Spins(\cC_0,-\nup)$.
\end{remark}

Our goal now is to show that the image of the homomorphism
 $\theta$ in \eqref{eq:theta} is the automorphism group scheme
$\AAut(\cC,\{...\})=\AAut(\cC,X^\cC_1)$ (see \eqref{eq:autos_pair83}).

\begin{remark}\label{re:left_mult}
Composing the isomorphisms $\Phi$ in \eqref{eq:Phi} and $\Psi$ in \eqref{eq:Psi}, for any $x\in\cC_0$ we obtain:
\[
\Phi\Psi(x)=\Phi(e_0\cdot x)=
\begin{pmatrix} 0&l_{e_0}\\ r_{e_0}&0\end{pmatrix}
\begin{pmatrix} 0&l_x\\ r_x&0\end{pmatrix}
=\begin{pmatrix} l_{e_0}r_x&0\\ 0&r_{e_0}l_x\end{pmatrix}
=\begin{pmatrix} L_x&0\\ 0&R_x\end{pmatrix}
\]
where $L_x:y\mapsto xy$, $R_x:y\mapsto yx$, are the operators of 
left and right multiplication by $x$ in $\cC$. This follows from 
$l_{e_0}r_x(y)=l_{e_0}(\bar y\bar x)=\overline{\bar y\bar x}=xy$ and from $r_{e_0}l_x(y)=r_{e_0}(\bar x\bar y)=\overline{\bar x\bar y}
=yx$. 

Thus for $x_1,\ldots,x_{2s}\in\cC_0$, with 
$\prod_{i=1}^{2s}\nup(x_i)=1$, the image by $\theta$ in \eqref{eq:theta} of $u=x_1\cdots\cdot x_{2s}\in\Spin(\cC_0,\nup)$ is  $L_{x_1}\cdots L_{x_{2s}}$. These elements are shown in 
\cite[Theorem 10]{Korean} and, in a different way, in \cite[Corollary 2.5]{Eld00}, to exhaust the elements in $\Aut(\cC,\{...\})$.
\end{remark}

We are now ready for the computation of the affine group schemes in 
\eqref{eq:autos_pair83}. This result extends the results in \cite{Shaw2} and \cite{Korean}, where only the groups of rational points were computed. It also gives a new perspective on the automorphisms of $3$-fold cross products.

\begin{theorem}\label{th:autos_83}
The group scheme of automorphisms $\AAut(\cC,\{...\})$ is isomorphic to $\Spins(\cC_0,-\nup)$.
\end{theorem}

\begin{proof}
The homomorphism $\theta$ in \eqref{eq:theta} factors through
$\AAut(\cC,\{...\})$ because if a triple
$(f_0,f_1,f_2)\in\Os^+(\cC,\nup)^3(R)$, 
for an $\FF$-algebra $R$, satisfies  
$f_0(x\bullet y)=f_1(x)\bullet f_2(y)$ for all $x,y,z\in \cC_R$, and 
with $f_0(e_0)=e_0$, then we have $f_1(\bar x)=\overline{f_2(x)}$
and 
\[
\begin{split}
f_2(\{xyz\})&=f_2\bigl((x\bar y)z\bigr)\\
&=f_2\bigl((y\bar x)\bullet \bar z\bigr)=f_0(y\bar x)\bullet f_1(\bar z)\\
&=f_0\bigl(\bar y\bullet x\bigr)\bullet \overline{f_2(z)}\\
&=\bigl(f_1(\bar y)\bullet f_2(x)\bigr)\bullet \overline{f_2(z)}\\
&=\bigl(\overline{f_2(y)}\bullet f_2(x)\bigr)\bullet\overline{f_2(z)}\\
&=\bigl(\overline{f_2(y)\overline{f_2(x)}}\bigr)f_2(z)\\
&=\bigl(f_2(x)\overline{f_2(y)}\bigr)f_2(z)=\{f_2(x)f_2(y)f_2(z)\},
\end{split}
\]
and hence we get an injective morphism 
$\theta:\Spins(\cC_0,-\nup)\rightarrow \AAut(\cC,\{...\})$ (which we 
denote also by $\theta$).

Remark \ref{re:left_mult} and \cite[Theorem 10]{Korean} show that $\theta_{\overline{\FF}}$ is bijective, where $\overline{\FF}$ denotes an
algebraic closure of $\FF$. Also, the differential $\dc\theta$ is injective,  because $\theta$ in \eqref{eq:theta} is injective. But the Lie algebra $\Lie\bigl(\AAut(\cC,\{...\})\bigr)=\Der(\cC,\{...\})$ is isomorphic to the orthogonal Lie algebra $\frso(\cC_0,-\nup)$ (\cite[Theorem 12]{Korean}), so by dimension count $\dc\theta$ is an isomorphism 
$\Lie\bigl(\Spins(\cC_0,-\nup)\bigr)\rightarrow 
\Lie\bigl(\AAut(\cC,\{...\})\bigr)$. We conclude from \cite[Theorem A.50]{EKmon} that $\theta:\Spins(\cC_0,-\nup)\rightarrow \AAut(\cC,\{...\})$ is an isomorphism.
\end{proof}

\bigskip


\section{Gradings}\label{se:gradings}

Let $G$ be an abelian group and let $(V,X)$ be an $r$-fold cross product. A 
\emph{$G$-grading} on $(V,X)$ is a vector space decomposition
$\Gamma: V=\bigoplus_{g\in G}V_g$, such that 
\[
X\bigl(V_{g_1},\ldots,V_{g_r}\bigr)\subseteq V_{g_1\cdots g_r}
\]
for all $g_1,\ldots,g_r\in G$. The subspaces $V_g$ are called the \emph{homogeneous components} and we write $\deg(v)=g$ (or 
$\deg_\Gamma(v)=g$) for any $0\neq v\in V_g$. The 
\emph{support} of $\Gamma$ is the (finite) subset 
$\supp\Gamma\bydef\{g\in G\mid V_g\neq 0\}$.

Two $G$-gradings $\Gamma:V=\bigoplus_{g\in G}V_g$ and 
$\Gamma':V=\bigoplus_{g\in G}V'_g$ on $(V,X)$  are \emph{isomorphic} if there is an automorphism $\varphi\in\Aut(V,X)$ such that $\varphi(V_g)=V'_g$ for all $g\in G$.

For the basic facts on gradings on algebras, the reader is referred to the monograph \cite{EKmon}. Here we will review the main facts, adapted to cross products.

\smallskip

Any $G$-grading $\Gamma$ on $(V,X)$ determines a homomorphism 
of affine group schemes 
\[
\rho_\Gamma: G\subdiag \longrightarrow \AAut(V,X)
\]
where $G\subdiag$ is the \emph{diagonalizable} affine group 
scheme represented by the group algebra $\FF G$. This is the Cartier dual of the constant group scheme $G$ (see \cite[\S 20]{KMRT}).

The homomorphism $\rho_\Gamma$ factors through the subgroup scheme $\AAut(V,X,\bup)$ if and only if $\bup(V_g,V_h)=0$ unless $gh=e$ (the neutral element of $G$). In this case, $\Gamma$ will be said to be \emph{compatible with $\bup$}.

Two $G$-gradings on $(V,X)$ are isomorphic if and only if the corresponding homomorphisms are conjugate by an element in $\Aut(V,X)$ \cite[Proposition 1.36]{EKmon}.

\smallskip

Conversely, any homomorphism $\rho:G\subdiag\rightarrow \AAut(V,X)$
determines a grading $\Gamma$ on $(V,X)$ as follows. Take the generic element (the identity map) $\id_{\FF G}\in G\subdiag(G)=\Hom(\FF G,\FF G)$ (algebra homomorphisms). The automorphism 
\[
\rho_{\FF G}(\id_{\FF G})\in\Aut(V_{\FF G},X_{\FF G})
\] 
will be called the
\emph{generic automorphism} attached to $\rho$, and 
$\Gamma: V=\bigoplus_{g\in G}V_g$ is given by:
\[
V_g=\{v\in V\mid \rho_{\FF G}(\id_{\FF G})(v\otimes 1)=v\otimes g\}.
\]
We will say that $\rho_{\FF G}(\id_{\FF G})$ is the \emph{generic
automorphism} of $\Gamma$ and will be denoted by 
$\varphi_\Gamma$.

\smallskip

Let $\Gamma: V=\bigoplus_{g\in G}V_g$ be a $G$-grading on the 
$r$-fold cross product $(V,X)$.
The \emph{diagonal group scheme} $\Diags(\Gamma)$ is the 
diagonalizable group scheme whose group of $R$-points, for any $\FF$-algebra $R$, consists of those automorphisms of the scalar extension 
$(V_R,X_R)$ that act by a scalar on each homogeneous component:
\[
\Diags(\Gamma)(R)=\{f\in\Aut(V_R,X_R)\mid
f\vert_{V_g\otimes_\FF R}\in R^\times\id_{V_g\otimes_\FF R}\}.
\]
Up to isomorphism, $\Diags(\Gamma)$ is $U\subdiag$ for a finitely generated abelian group $U$, which is called the \emph{universal grading group} of $\Gamma$.  The morphism $\rho_\Gamma$ factors through $\Diags(\Gamma)$ and hence induces a natural group
homomorphism $U\rightarrow G$. Moreover, $\Gamma$ may be considered, 
in a natural way, a grading by $U$, which is the most natural grading group for $\Gamma$.

\smallskip

Given a $G$-grading $\Gamma:V=\bigoplus_{g\in G}V_g$ and an $H$-grading 
$\Gamma':V=\bigoplus_{h\in H}V'_h$ on $(V,X)$, $\Gamma$ is said to 
be a \emph{refinement} of $\Gamma'$, or $\Gamma'$ a 
\emph{coarsening} of $\Gamma$, if for any $g\in G$ there is an 
element $h\in H$ such that $V_g\subseteq V'_h$.   In other words, 
each homogeneous component for $\Gamma'$ is a sum of 
homogeneous components for $\Gamma$. The refinement is said to be 
proper if we have that $V_g$ is strictly contained in $V'_h$ for at least 
one nonzero homogeneous component. In particular, given a $G$-grading $\Gamma:V=\bigoplus_{g\in G}V_g$ and a group homomorphism $\alpha:G\rightarrow H$, the grading
$\Gamma^\alpha:V=\bigoplus_{h\in H}V_h'$, where 
$V_h'=\bigoplus_{\alpha(g)=h}V_g$ for all $h\in H$, is a coarsening of $\Gamma$. Any homogeneous element for $\Gamma$ of degree $g$ is homogeneous too for $\Gamma^\alpha$ of degree $\alpha(g)$.

A grading $\Gamma$ is said to be \emph{fine} if it admits no proper refinement. In that case, $\Diags(\Gamma)$ is a maximal diagonalizable subgroup scheme of $\AAut(V,X)$.  

Thus the classification of fine gradings \emph{up to equivalence} is equivalent to the classification of maximal diagonalizable subgroup schemes of 
$\AAut(V,X)$ up to conjugation by elements in $\Aut(V,X)$.

Here two gradings  $\Gamma:V=\bigoplus_{g\in G}V_g$ and 
$\Gamma':V=\bigoplus_{h\in H}V'_h$ on $(V,X)$  are \emph{equivalent} if there is an automorphism $\varphi\in\Aut(V,X)$ such that for any $g\in G$ there is $h\in H$ such that $\varphi(V_g)=V'_h$. 

Equivalence is weaker than isomorphism. For equivalence, the diagonal group schemes are imposed to be conjugate by an automorphism of 
$(V,X)$, while for isomorphism, we require that the morphisms of affine group schemes $\rho_\Gamma$  are conjugate by an automorphism.

Given a grading $\Gamma: V=\bigoplus_{g\in G}V_g$, the automorphism group of $\Gamma$, $\Aut(\Gamma)$ is the group of
 self-equivalences, that is, of automorphisms of $(V,X)$ that permute
 the homogeneous components. The stabilizer of $\Gamma$: 
$\Stab(\Gamma)$, is the kernel of the natural induced map 
$\Aut(\Gamma)\rightarrow \Sym\bigl(\supp(\Gamma)\bigr)$, where 
$\Sym(S)$ denotes the symmetric group on the set $S$. That is, 
$\Stab(\Gamma)$ is the subgroup of self-isomorphisms of $\Gamma$.
 The quotient $W(\Gamma)\bydef\Aut(\Gamma)/\Stab(\Gamma)$ is called the \emph{Weyl group}.

\medskip

In this section, \emph{the ground field $\FF$ will be assumed to be algebraically closed} (and, as always, of characteristic not two), 
unless otherwise stated. 
All gradings by abelian groups on cross products will be classified up to isomorphism. The fine gradings will be classified up to equivalence too, thus obtaining the maximal diagonalizable subgroup schemes of the corresponding automorphism group schemes.

\bigskip

\subsection{$n$ even, $r=1$}\quad 

Let $X:V\rightarrow V$ be a $1$-fold cross product on the even-dimensional vector space $V$ over an algebraically closed field $\FF$.

Let $G$ be an abelian group, a $G$-grading 
$\Gamma: V=\bigoplus_{g\in G}V_g$ on $(V,X)$ is nothing else than a vector space decomposition of $V$ on subspaces invariant under the endomorphism $X:V\rightarrow V$. 

Hence, if $V_+$ and $V_-$ are the subspaces defined in Corollary \ref{co:2s1}, a $G$-grading on $(V,X)$ determines $G$-gradings on 
the vector subspaces $V_+$ and $V_-$ and, conversely, any pair of 
$G$-gradings on $V_+$ and $V_-$ (as vector spaces) determines a $G$-grading on $(V,X)$.  

In particular, up to equivalence, there is a unique fine grading $\Gamma$, with universal group $\ZZ^n$, $n=2s$, obtained by fixing bases $\{v_1,\ldots, v_s\}$, $\{w_1,\ldots,w_s\}$, of $V_+$ and $V_-$, and assigning
degrees as follows:
\[
\deg(v_i)=\epsilon_i,\quad \deg(w_i)=\epsilon_{s+i},\quad 
i=1,\ldots, s,
\]
where $\epsilon_1,\ldots,\epsilon_n$ is the canonical basis of $\ZZ^n$.
The Weyl group $W(\Gamma)$ is $\Sym_s\times\Sym_s$, where 
$\Sym_s$ is the symmetric group on $s$ symbols, because any self-equivalence permutes the homogeneous components of $V_+$ and $V_-$ separately.

\begin{remark}\label{re:gr_2s1}
The above is valid for arbitrary fields $\FF$ with $-1\in\FF^2$, and implies that a fine grading on $(V,X)$ is given by a pair of fine gradings on $V_+$ and $V_-$ as vector spaces. In particular, there is a unique fine grading up to equivalence, and the corresponding maximal diagonalizable subgroup scheme is just a maximal torus of 
$\AAut(V,X)=\GLs(V_+)\times\GLs(V_-)$. (All maximal tori are conjugate.)

If $-1\not\in\FF^2$, then $\KK=\FF\id\oplus\FF X$ is a field. Hence any homogeneous component of a grading $\Gamma:V=\bigoplus_{g\in G}V_g$ on $(V,X)$ is a vector space over $\KK$, so $\dim_\FF V_g$ is even for any $g\in G$. In particular, $\Gamma$ is a coarsening of a grading by $\ZZ^s$ ($n=2s$) obtained by taking a $\KK$-basis $\{v_1,\ldots,v_s\}$ of $V$ and declaring that $\deg(v_i)=\epsilon_i$, where the $\epsilon_i$'s are the canonical generators of $\ZZ^s$.
Hence again there is a unique fine grading, up to equivalence, and the corresponding diagonalizable group scheme is a maximal split torus of $\AAut(V,X)=\Cs_{\GLs(V)}(X)$ (see Theorem \ref{th:2s1}).
\end{remark}

\bigskip

\subsection{$n\geq 3$, $r=n-1$}\quad

Let $(V,X)$ be an $r$-fold cross product with $\dim_\FF V=n\geq 3$ and $r=n-1$, over an algebraically closed field $\FF$, relative to a 
nondegenerate symmetric bilinear form $\bup$, unique up to a scalar (Proposition \ref{pr:uniqueness_b}). Equation \eqref{eq:Xsigma} shows that $X$ is unique, up to isomorphism.

Let $G$ be an abelian group and let $\Gamma:V=\bigoplus_{g\in G}V_g$ be a $G$-grading on $(V,X)$. The generic automorphism
$\varphi_\Gamma$ is in $\AAut(V,X)(\FF G)=\widetilde\Os(V,\bup)(\FF G)$.

Therefore, for any $g_1,g_2\in G$ and $v_1\in V_{g_1}$,
$v_2\in V_{g_2}$, on the group algebra $\FF G$ we have
\[
\bup(v_1,v_2)g_1g_2=\bup\bigl(v_1\otimes g_1,v_2\otimes g_2\bigr)
=\det(\varphi_\Gamma)\bup(v_1,v_2),
\]
and hence we get
\begin{equation}\label{eq:orthogonality}
\bup(V_{g_1},V_{g_2})=0\quad \text{unless
$g_1g_2=\det(\varphi_\Gamma)$.}
\end{equation}
Write  $h=\det(\varphi_\Gamma)$ and
consider the map
\[
\begin{split}
\delta: G&\longrightarrow \ZZ_{\geq 0}\\
 g\,&\mapsto \delta(g)\bydef \dim_\FF V_g,
\end{split}
\]
which satisfies the following restrictions:
\begin{itemize}
\item $\sum_{g\in G}\delta(g)=n$,
\item $h=\det(\varphi_\Gamma)=\prod_{g\in G}g^{\delta(g)}$,
\item $\delta(g)=\delta(g^{-1}h)$ for all $g\in G$. This is a consequence of \eqref{eq:orthogonality}.
\end{itemize}
By Proposition \ref{pr:exact}, the determinant 
$h=\det(\varphi_\Gamma)$ satisfies $h^{n-2}=e$. ($e$ denotes the 
neutral element of $G$.) 

Note that $h$ is the neutral element if and only if $\Gamma$ is compatible with $\bup$. Also, because of \eqref{eq:orthogonality}, for an element $g\in\supp\Gamma$, the restriction $\bup\vert_{V_g}$ is nondegenerate if and only if $g^2=h$, otherwise $V_g$ is totally isotropic.

Conversely, given a map $\delta$ satisfying the restrictions above, we can always define a $G$-grading on $(V,X)$. For example, let $n=5$, $g_1,g_2,h\in G$ with $h^3=e$, $g_1^2=h$, $g_2^2\neq h$. 
(This implies that the elements $g_1$, $g_2$ and $g_2^{-1}h$ are different.) Assume that $\delta(g_1)=1$, 
$\delta(g_2)=2=\delta(g_2^{-1}h)$, and $\delta(g)=0$ for any 
$g\neq g_1,g_2,g_2^{-1}h$. Since $\FF$ is algebraically closed we can take a basis $\{v_1,v_2,v_3,v_4,v_5\}$ of $V$ with 
$\bup(v_1,v_1)=1=\bup(v_2,v_4)=\bup(v_3,v_5)$, and all the other values of $\bup(v_i,v_j)$ equal to $0$. Then with $V_{g_1}=\FF v_1$, $V_{g_2}=\FF v_2+\FF v_3$, $V_{g_2^{-1}h}=\FF v_4+\FF v_5$, 
and $V_g=0$ for $g\neq g_1,g_2,g_2^{-1}h$, we get a $G$-grading on 
$(V,X)$, as its generic automorphism lies in $\widetilde\Os(V,\bup)$. 

Up to isomorphism, the $G$-grading associated to a map $\delta$ satisfying the restrictions above is unique, and will be denoted by 
$\Gamma(G,\delta)$.

We have proved our next result:

\begin{theorem}\label{th:grn-1}
Let $X:V^{n-1}\rightarrow V$ be an ($n-1$)-fold cross product on the 
$n$-dimensional vector space $V$ ($n\geq 3$) over an algebraically closed field $\FF$, relative to a 
nondegenerate symmetric bilinear form $\bup$. Let $G$ be an abelian group.

Then any $G$-grading on $(V,X)$ is isomorphic to a grading $\Gamma(G,\delta)$ for a unique map $\delta:G\rightarrow \ZZ_{\geq 0}$ satisfying the restrictions:
\begin{itemize}
\item $\sum_{g\in G}\delta(g)=n$,
\item $\delta(g)=\delta(g^{-1}h)$ for all $g\in G$, where 
$h=\prod_{g\in G}g^{\delta(g)}$. 
\end{itemize}
\end{theorem}

Let $\delta$ be a map satisfying the restrictions in Theorem \ref{th:grn-1}, and let $\Gamma=\Gamma(G,\delta)$ be the 
associated grading. Then, joining suitable bases of the homogeneous components, we obtain a basis of $V$ consisting of homogeneous elements: 
\[
\{u_1,\ldots,u_p,v_1,w_1,\ldots,v_q,w_q\}\quad (p,q\geq 0,\  n=p+2q),
\] 
with $\bup(u_i,u_i)=1$ for all $i=1,\ldots,p$, $\bup(v_j,w_j)=1\bigr(=\bup(w_j,v_j)\bigr)$ for all $j=1,\ldots,q$, and all other values of $\bup$ being $0$, and with 
\[
\deg(u_i)=g_i\ \text{for $i=1,\ldots,p$,}\quad
 \deg(v_j)=g'_j,\ \deg(w_j)=g_j''\ \text{for $j=1,\ldots,q$,}
\]
 for group elements $g_1,\ldots,g_p,g_1',g_1'',\ldots,g_q',g_q''$ satisfying 
\[
g_i^2=h=g_j'g_j'',\quad \text{where $h\bydef g_1\cdots g_pg_1'g_1''\cdots g_q'g_q''$.}
\] 
Note that $n=p+2q$.

Let $U$ be the abelian group, defined by generators and relations, with generators $x_1,\ldots,x_p,y_1,z_1,\ldots, y_q,z_q$, subject to the relations:
\begin{equation}\label{eq:relations}
x_1^2=\cdots=x_p^2=y_1z_1=\cdots =y_qz_q=x_1\cdots x_py_1z_1\cdots y_qz_q\,.
\end{equation}
(We will write $U=\langle x_1,\ldots,x_p,y_1,z_1,\ldots,y_p,z_p\mid 
x_1^2=\cdots=x_p^2=y_1z_1=\cdots =y_qz_q
=x_1\cdots x_py_1z_1\cdots y_qz_q\rangle$.)

Also, let $\delta_U:U\rightarrow \ZZ_{\geq 0}$ be the map defined by 
\[
\delta_U(x_i)=\delta_U(y_j)=\delta_U(z_j)=1
\]
for all $i=1,\ldots,p$, $j=1,\ldots,q$, and $\delta_U(u)=0$ for any other element $u\in U$.

This is well defined and satisfies the conditions of Theorem \ref{th:grn-1}
because the elements $x_1,\ldots,x_p,y_1,z_1,\ldots, y_q,z_q$ are all different. Actually, we get the following possibilities:
\begin{itemize}
\item If $p=0$,  then $n=2q$ is even with $q\geq 2$, and 
\[
U=\langle y_1,z_1,\ldots,y_q,z_q\mid y_1z_1=\cdots = y_qz_q=y_1z_1\cdots y_qz_q\rangle.
\]

With $u=y_1z_1=\cdots=y_qz_q=y_1z_1\cdots y_qz_q$, we have 
$u=u^q$. Then $U$ is generated by the elements $u,y_1,\ldots,y_q$ 
and is isomorphic to $\ZZ^q\times \ZZ/(q-1)$ by means of the map 
that takes $(0,\ldots,1,\ldots,0;\bar 0)$  to $y_j$, and 
$(0,\ldots,-1,\ldots,0;\bar 1)$ to $z_j$ ($1$ in the $j$th position). 
This shows that the elements $y_i,z_i$ are all different.

\item If $p=1$, then $q\geq 1$, $n=1+2q$ is odd, and
\[
U=\langle x_1,y_1,z_1,\ldots,y_q,z_q\mid x_1^2=y_1z_1=\cdots = y_qz_q=x_1y_1z_1\cdots y_qz_q\rangle.
\]
Thus $x_1^2=x_1(x_1^2)^q=x_1^{1+2q}$, so $x_1^{2q-1}=e$. If 
$q=1$, we get $x_1=e$ and $U$ is isomorphic to $\ZZ$. If $q\geq 2$, 
then $U$ is generated by $x_1,y_1,\ldots,y_q$ and it is isomorphic to 
$\ZZ^q\times \ZZ/(2q-1)$ by means of the map that takes 
$(0,\ldots,0 ;\bar 1)$ to $x_1$, $(0,\ldots,1,\ldots,0;\bar 0)$  to 
$y_j$, and $(0,\ldots,-1,\ldots,0 ;\bar 2)$ to $z_j$ ($\pm 1$ in the $j$th position).

\item Finally, if $p\geq 2$, write $x_i=x_1t_i$, with $t_i^2=e$, for $i=2,\ldots p$. The relations \eqref{eq:relations} imply 
$x_1^2=x_1\cdots x_py_1z_1\cdots  y_qz_q=x_1^{p+2q}t_2\cdots t_p$, so $t_2\cdots t_p=x_1^{p+2q-2}=x_1^{n-2}$ and $x_1^{2n-4}=e$. Then $U$ is generated by the elements $x_1,t_2,\ldots,t_{p-1},y_1,\ldots,y_q$ and it is isomorphic to $\ZZ^q\times \ZZ/(2n-4)\times\left(\ZZ/2\right)^{p-2}$ by means of the map that takes 
\begin{align*}
(0,\ldots,0;\bar 1;\bar 0,\ldots,\bar 0)&\mapsto x_1 & 
    (0,\ldots,1,\ldots,0;\bar 0;\bar 0,\ldots,\bar 0)&\mapsto y_j\\
(0,\ldots,0;\bar 1;\bar 0,\ldots,\bar 1,\ldots,\bar 0)&\mapsto x_i&
    (0,\ldots,-1,\ldots,0;\bar 2;\bar 0,\ldots,\bar 0)&\mapsto z_j\\
(0,\ldots,0;\overline{n-1};\bar 1,\ldots,\bar 1)&\mapsto x_p&&
\end{align*}
for $2\leq i\leq p-1$, $1\leq j\leq q$.
\end{itemize}

The map $\delta_U$ satisfies the restrictions in Theorem \ref{th:grn-1}, so it determines a grading $\Gamma(U,\delta_U)$, with one-dimensional homogeneous components, and hence fine, that refines $\Gamma(G,\delta)$. Moreover, $U$ is its universal grading group.

This gives the classification of fine gradings, up to equivalence:

\begin{corollary}\label{co:finegrn-1}
Let $X:V^{n-1}\rightarrow V$ be an ($n-1$)-fold cross product on the 
$n$-dimensional vector space $V$ ($n\geq 3$) over an algebraically closed field $\FF$, relative to a 
nondegenerate symmetric bilinear form $\bup$. 

Up to equivalence, the fine gradings on $(V,X)$ are the gradings $\Gamma(U,\delta_U)$, where $U$ is the abelian group
\begin{multline*}
U=\langle x_1,\ldots,x_p,y_1,z_1,\ldots,y_p,z_p\mid \\
x_1^2=\cdots=x_p^2=y_1z_1=\cdots =y_qz_q
=x_1\cdots x_py_1z_1\cdots y_qz_q\rangle
\end{multline*}
with $p+2q=n$, and $\delta_U:U\mapsto \ZZ_{\geq 0}$ is the map given by
\[
\delta_U(x_i)=\delta_U(y_j)=\delta_U(z_j)=1
\]
for all $i=1,\ldots,p$, $j=1,\ldots,q$, and $\delta_U(u)=0$ for any other element $u\in U$.

Moreover, $U$ is, up to isomorphism, the universal grading group of 
$\Gamma(U,\delta_U)$ and the following conditions hold:
\begin{itemize}
\item if $p=0$, $U$ is isomorphic to $\ZZ^q\times \ZZ/(q-1)$,
\item if $p=1$, $U$ is isomorphic to $\ZZ^q\times \ZZ/(2q-1)$,
\item if $p>1$, $U$ is isomorphic to $\ZZ^q\times \ZZ/(2n-4)\times
\left(\ZZ/2\right)^{p-2}$.
\end{itemize}
\end{corollary}

Any self-equivalence of a fine grading $\Gamma(U,\delta_U)$ in Corollary \ref{co:finegrn-1} permutes the homogeneous spaces of degrees $x_i$, $1\leq i\leq p$, and the pairs of homogeneous spaces of degrees $y_i$ and $z_i$. Therefore, the Weyl group 
 is the Cartesian product of the symmetric group on $p$ symbols, and the `signed permutation group' on $q$ symbols:
\[  
W\bigl(\Gamma(U,\delta_U)\bigr)=\Sym_p\times\left(\left(\ZZ/2\right)^q\rtimes\Sym_q\right).
\]

\begin{example}\label{ex:quaternions_3fold}
Let $\cQ$ be the algebra of quaternions over the algebraically closed field $\FF$, with norm $\nup$, which is multiplicative. Up to isomorphism $\cQ$ is the algebra of $2\times 2$ matrices over $\FF$, and the norm is given by the determinant. Let $x\mapsto \bar x=\nup(x,1)1-x$ be the canonical involution. It satisfies $x\bar x=\nup(x)1=\bar x x$ for all $x\in\cQ$. Define the multilinear map 
\[
\begin{split}
X:\cQ^3&\longrightarrow\cQ\\
(x,y,z)&\mapsto X(x,y,z)\bydef x\bar yz-z\bar yx.
\end{split}
\]
This map is alternating, because $X(x,x,z)=x\bar xz-z\bar x x=\nup(x)z-\nup(x)z=0$, and $X(x,y,y)=x\bar y y-y\bar y x=\nup(y)x-\nup(y)x=0$.
Besides, for all $x,y,z\in\cQ$, $\nup\bigl(X(x,y,z),x\bigr)=\nup(x\bar yz,x)-\nup(z\bar yx,x)=\nup(x)\nup(\bar yz-z\bar y,1)=0$.

It turns out that $X$ is a $3$-fold cross product on $\cQ$. Indeed, for any $x,y,z\in\cQ$ and $\lambda\in\FF$, the coefficient of $\lambda^2$ in both sides of the equality $\nup\bigl((x+\lambda z)\bar y(x+\lambda z)\bigr)=\nup(x+\lambda z)^2\nup(y)$, gives
\begin{equation}\label{eq:Qa}
\nup(x\bar yz+z\bar yx)+\nup(x\bar yx,z\bar yz)=
\bigl(\nup(x,z)^2+2\nup(x)\nup(z)\bigr)\nup(y).
\end{equation}
Also we have
\begin{equation}\label{eq:Qb}
\begin{split}
\nup(x\bar yz+z\bar yx)
&=\nup(x\bar yz)+\nup(z\bar yx)+\nup(x\bar yz,z\bar y x)\\
  &=2\nup(x)\nup(y)\nup(z)+\nup(x\bar yz,z\bar yx),
\end{split}
\end{equation}
and
\begin{equation}\label{eq:Qc}
\begin{split}
\nup(x\bar y x,z\bar yz)
&=\nup\bigl(\nup(x,y)x-\nup(x)y,\nup(y,z)z-\nup(z)y\bigr)\\
&=\nup(x,y)\nup(y,z)\nup(z,x)-\nup(x)\nup(y,z)^2\\
&\qquad\qquad -\nup(z)\nup(x,y)^2+2\nup(x)\nup(y)\nup(z).
\end{split}
\end{equation}
Putting all these together, we get:
\[
\begin{split}
\nup\bigl(X(x,y,z),&X(x,y,z)\bigr)\\
&=
\nup(x\bar yz-z\bar yx,x\bar yz-z\bar yx)\\
 &=\nup(x\bar yz,x\bar yz)+\nup(z\bar yx,z\bar yx)
    -2\nup(x\bar yz,z\bar yx)\\
&=8\nup(x)\nup(y)\nup(z)-2\nup(x\bar yz+z\bar yx)\quad\text{(by \eqref{eq:Qb})}\\
&=4\nup(x)\nup(y)\nup(z)
+2\nup(x\bar yx,z\bar yz)-2\nup(x,z)^2\nup(y)
\quad\text{(by \eqref{eq:Qa})}\\
&=\begin{vmatrix} \nup(x,x)&\nup(x,y)&\nup(x,z)\\
  \nup(y,x)&\nup(y,y)&\nup(y,z)\\
  \nup(z,x)&\nup(z,y)&\nup(z,z)\end{vmatrix}\quad\text{(by \eqref{eq:Qc}),}
\end{split}
\]
so that $X$ turns out to be a $3$-fold cross product relative to the polar form $\nup(.,.)$.

Corollary \ref{co:finegrn-1} shows that, up to equivalence, there are three different fine gradings on $(\cQ,X)$ with universal groups $\ZZ^2$ ($p=0$, $q=2$), $\ZZ\times \ZZ/4$ ($p=2$, $q=1$), and $\ZZ/4\times\bigl(\ZZ/2\bigr)^2$ ($p=4$, $q=0$).

In contrast to this, the quaternion algebra $\cQ$ has only two fine gradings, with universal groups $\ZZ$ and $\bigl(\ZZ/2\bigr)^2$ (see, e.g., \cite[Example 2.40]{EKmon}).
\end{example}

\bigskip

\subsection{$n=7$, $r=2$}\quad 

Any $2$-fold cross product on a
seven-dimensional vector space over an algebraically closed field $\FF$ is isomorphic to $(\cC_0,X^{\cC_0})$ for the unique Cayley algebra 
$\cC$ over $\FF$  (Theorem \ref{th:CP}).

Given an abelian group $G$, any $G$-grading on $(\cC_0,X^{\cC_0})$
extends uniquely to a grading on the Cayley algebra $\cC$, and two such gradings on $(\cC_0,X^{\cC_0})$ are isomorphic if and only if so 
are the extended gradings on $\cC$ (see \cite{EldOctonions} or 
\cite[Corollary 4.25]{EKmon}). 

Moreover, the classification of $G$-gradings, up to isomorphism, on 
the Cayley algebra $\cC$ is given in \cite[Theorem 4.21]{EKmon}, based on \cite{EldOctonions}. There are two types of gradings, those obtained as a coarsening of the $\ZZ^2$-grading given by the weight
space decomposition relative to a maximal torus of $\AAut(\cC)$, and those equivalent to the $\left(\ZZ/2\right)^3$-grading obtained by 
means of the Cayley-Dickson doubling process. 

In particular, there exist only two fine gradings, up to equivalence, with
universal group $\ZZ^2$ and $\left(\ZZ/2\right)^3$, respectively. The respective Weyl groups are the Weyl group of the root system of type $G_2$ (i.e., the dihedral group of order $12$) and the general linear
 group $\GL_3(\FF_2)$ of degree $3$ over the field of two elements 
(see \cite[Theorems 4.17 and 4.19]{EKmon}).

\bigskip

\subsection{$n=8$, $r=3$}\quad

Any $3$-fold cross product on an eight-dimensional vector space over
an algebraically closed field $\FF$ is isomorphic, by Theorem \ref{th:CP} and using that $\FF$ is algebraically closed, to 
$(\cC,X_1^{\cC})$ for the unique Cayley algebra $\cC$ over $\FF$, where $X_1^\cC$ is given by Equation \eqref{eq:typeI}.

Equation \eqref{eq:autos_pair83} shows that the group schemes 
$\AAut(\cC,X_1^\cC)$ and $\AAut(\cC,\{...\})$ coincide, where the triple product $\{...\}$ is given in \eqref{eq:3C}:
$\{xyz\}\bydef (x\bar y)z$ for $x,y,z\in\cC$. Moreover, Theorem \ref{th:autos_83} and its proof show that $\AAut(\cC,\{...\})$ is isomorphic to $\Spins(\cC_0,-\nup)$ and is contained in the special orthogonal group scheme $\Os^+(\cC,\nup)$, where $\nup$ denotes the norm of the Cayley algebra $\cC$.

Therefore, in order to study gradings on the cross product $(\cC,X_1^\cC)$, it is enough to study gradings on $(\cC,\{...\})$.

\smallskip

We need some previous results. Continuing the discussion in Remark \ref{re:left_mult}, \cite[Theorem 10]{Korean} and 
\cite[Corollary 2.5]{Eld00} give:
\begin{equation}\label{eq:Spin7}
\Aut(\cC,\{...\})=\left\{\prod_{i=1}^mL_{x_i}\mid m\geq 0,\, 
x_i\in\cC_0\ \forall i,\ \text{and}\ \prod_{i=1}^m\nup(x_i)=1\right\},
\end{equation}
where $L_x:y\mapsto xy$ denotes the left multiplication by $x$ in $\cC$.
Take a \emph{standard basis}
$\{e_1,e_2,u_1,u_2,u_3,v_1,v_2,v_3\}$  of $\cC$ (see, e.g., \cite[p.~41]{EKmon}), with multiplication given by:
\begin{itemize}
\item $e_1$ and $e_2$ are orthogonal idempotents:
$e_1^2=e_1$, $e_2^2=e_2$, $e_1e_2=e_2e_1=0$,
\item $e_1u_i=u_ie_2=u_i$, $e_2u_i=u_ie_1=0$, for all $i=1,2,3$,
\item $e_2v_i=v_ie_1=v_i$, $e_1v_i=v_ie_2=0$, for all $i=1,2,3$,
\item $u_iv_i=-e_1$, $v_iu_i=-e_2$, for all $i=1,2,3$,
\item $u_iu_{i+1}=-u_{i+1}u_i=v_{i+2}$,
     $v_iv_{i+1}=-v_{i+1}v_i=u_{i+2}$, indices modulo $3$,
\item all other products between basic elements are $0$.
\end{itemize}

\begin{lemma}\label{le:orbits_83}
Let $\cC$ be the Cayley algebra over an algebraically closed field $\FF$.
\begin{romanenumerate}
\item The orbit of $1$ under $\Aut(\cC,\{...\})$ is the `unit sphere': 
\[
\orbit_{\Aut(\cC,\{...\})}(1)=\{x\in\cC\mid \nup(x)=1\}.
\]

\item The orbit of $e_1$ is the set of nonzero isotropic elements:
\[
\orbit_{\Aut(\cC,\{...\})}(e_1)=\{x\in\cC\mid \nup(x)=0,\, x\neq 0\}.
\]

\item The orbit of the pair $(e_1,e_2)$ under the diagonal action of $\Aut(\cC,\{...\})$ on $\cC\times\cC$ is the set
 \[
\orbit_{\Aut(\cC,\{...\})}(e_1,e_2)=\{(x,y)\in\cC\times\cC
  \mid \nup(x)=0=\nup(y),\ \nup(x,y)=1\}.
\]
\end{romanenumerate}
\end{lemma}
\begin{proof}
As $\Aut(\cC,\{...\})$ is contained in the orthogonal group of the norm, it is clear that the orbit of $1$ is contained in the set of norm $1$ elements. Conversely, if $x\in\cC$ and $\nup(x)=1$, take $z\in\cC$ orthogonal to $1$ and $x$ of norm $1$. Then $x=(x\bar z)z=L_{x\bar z}L_z(1)$ and $L_{x\bar z}L_z\in\Aut(\cC,\{...\})$.

It is clear that the orbit of $e_1$ is contained in the set of nonzero elements of zero norm.

If $0\neq x\in\cC_0$ is an element with $\nup(x)=0$,  take 
$y\in\cC_0$ with $\nup(y)=0$ and $\nup(x,y)=1$. Then $-xy-yx=
x\bar y+y\bar x=\nup(x,y)1=1$ and $(xy)(xy)=(xy)(-1-yx)=-xy$, as $y^2=0$ because of \eqref{eq:CH}. It turns out that $f_1=-xy$ and $f_2=-yx$ are orthogonal idempotents, and thus (see, e.g. \cite[pp.~128-129]{EKmon}) there is an automorphism of $\cC$ (hence also of $(\cC,\{...\})$) such that $\varphi(f_1)=e_1$, $\varphi(f_2)=e_2$.

Also, $f_2x=-yx^2=0$, $xf_1=-x^2y=0$, and hence $f_1x=xf_2=x$, so that $\varphi(x)\in\{z\in\cC: e_1z=ze_2=z\}=U\bydef\FF u_1+\FF u_2+\FF u_3$. But any $u\in  U$ is mapped to $-u_3$ by a suitable automorphism of $\cC$ that preserves $e_1$ and $e_2$. Hence we may assume that $\varphi(x)=-u_3$. Note that 
$L_{u_2+v_2}L_{u_1+v_1}(e_1)=-u_3$ and we conclude that $x$ is in the orbit of $e_1$.

If $\nup(x)=0$ but $x\not\in\cC_0$, take $y\in \cC_0$ orthogonal to $x$ and with $\nup(y)=1$. Then $L_y\in\Aut(\cC,\{...\})$ and $L_y(x)=yx$ satisfies $\overline{yx}=-\bar x y=-\nup(x,y)1+\bar yx=-yx$, so $yx\in\cC_0$. By the previous paragraph, $yx$ lies in the orbit of $e_1$, and so does $x$.

Finally, since $\Aut(\cC,\{...\})$ is contained in the orthogonal group, it
is clear that the orbit of $(e_1,e_2)$ is contained in $\{(x,y)\in\cC\times\cC
  \mid \nup(x)=0=\nup(y),\ \nup(x,y)=1\}$. Conversely, let $x,y\in\cC$ 
with $\nup(x)=0=\nup(y)$, $\nup(x,y)=1$, and take $z\in \cC$ 
orthogonal to $1$, $x$, and $y$, with $\nup(z)=1$, and take $t\in\cC$ 
orthogonal to $1$ and $z$ with $\nup(t)=1$. Let $a=zt^{-1}=z\bar t$, 
$b=t$. Then $a,b\in\cC_0$, $\nup(a)=\nup(b)=1$, so 
$L_aL_b\in\Aut(\cC,\{...\})$. 
Besides, use \eqref{eq:nbar} to get $\nup(L_aL_b(x),1)=
\nup(x,\bar b\bar a)=\nup(x,\overline{ab})=\nup(x,z)=0$, and 
$\nup(L_aL_b(y),1)=\nup(y,z)=0$ too. Therefore, we may assume that $x$ and $y$ are elements in $\cC_0$. With some of the arguments above we may then assume that $xy=-e_2$ and $yx=-e_1$, $x\in V=\bigoplus_{i=1}^3\FF v_i$, $y\in U=\bigoplus_{i=1}^3\FF u_i$, and even that $x=v_1$ and $y=u_1$. The operator $\varphi=L_{u_1-v_1}L_{e_1-e_2}$ is in $\Aut(\cC,\{...\})$ and $\varphi(v_1)=e_1$, 
$\varphi(u_1)=e_2$. The result follows.
\end{proof}

Let us define a few natural gradings on $(\cC,\{...\})$.

\begin{example}[Cartan grading]\label{ex:Cartan}
The following assignment of degrees in $\ZZ^3$ gives a grading on 
$(\cC,\{...\})$, which is called the \emph{Cartan grading}, as its homogeneous components are the weight spaces relative to a maximal torus $\mathbf{T}$ of $\AAut(\cC,\{...\})$:
\begin{align*}
\deg(u_1)&=(1,0,0)=-\deg(v_1),\\
\deg(u_2)&=(0,1,0)=-\deg(v_2),\\
\deg(u_3)&=(0,0,1)=-\deg(v_3),\\
\deg(e_2)&=(1,1,1)=-\deg(e_1).
\end{align*}
This grading, denoted by $\Gamma^\cC\subcartan$, is fine since all the homogeneous components have dimension $1$, and it is straightforward to check that $\ZZ^3$ is its universal grading group.

The maximal torus $\mathbf{T}$ is precisely 
$\Diags(\Gamma^\cC\subcartan)$, and we will identify the associated 
Weyl group $W(\Gamma^\cC\subcartan)$ with a subgroup of 
$\GL_3(\ZZ)=\Aut(\ZZ^3)$. Note that $\Aut(\Gamma)$ is the normalizer of $T=\mathbf{T}(\FF)$ in $\Aut(\cC,\{...\})$, while 
$\Stab(\Gamma)$ is its centralizer. The quotient $W(\Gamma^\cC\subcartan)$ is then the Weyl group of the root
 system of type $B_3$, which is the signed permutation group 
$\left(\ZZ/2\right)^3\rtimes\Sym_3$.

Given any abelian group $G$ and a group homomorphism $\alpha:\ZZ^3\rightarrow G$, denote by $\Gamma^\cC(G,\alpha)$ the $G$-grading 
on $(\cC,\{...\})$ obtained as a coarsening of 
$\Gamma^\cC\subcartan$ by means of $\alpha$. That is, the degree in 
$\Gamma^\cC(G,\alpha)$ is the image under $\alpha$ of the degree in $\Gamma^\cC\subcartan$ for any homogeneous element relative to 
$\Gamma^\cC\subcartan$.
\end{example}

There is another interesting basis $\{1,w_i\mid 1\leq i\leq 7\}$ of  $\cC$, with multiplication determined by 
\begin{equation}\label{eq:CDbasis}
\begin{split}
&w_i^2=-1\ \text{for all $i$},\\
&w_iw_{i+1}=-w_{i+1}w_i=w_{i+3}\
\text{(cyclically in $i,i+1,i+3$, indices modulo $7$)}.
\end{split}
\end{equation}
Thus, for example, $w_1w_2=w_4$, $w_4w_1=w_2$, ...

This will be called a \emph{Cayley-Dickson basis}, or \emph{CD-basis}
for short. Any CD-basis is an orthonormal basis relative to $\nup$.
For example, one can take $w_1=\bi(e_1-e_2)$, $w_2=u_1+v_1$, $w_3=u_2+v_2$, and then $w_4=w_1w_2=\bi(u_1-v_1)$, 
$w_5=u_3+v_3$, $w_6=\bi(u_3-v_3)$, $w_7=\bi(u_2-v_2)$, where
$\bi$ is a square root of $-1$.

\begin{example}\label{ex:gGH}
Let $H$ be an elementary $2$-subgroup of rank $3$  of an abelian group $G$, so that $H$ is isomorphic to $\left(\ZZ/2\right)^3$, and let $h_1,h_2,h_3$ be generators of $H$. Denote by 
$\Gamma^\cC(G,H)$ the grading on the Cayley algebra $\cC$ determined by $\deg(w_i)=h_i$ for $i=1,2,3$ (the elements $w_i$, $i=1,2,3$, generate the algebra $\cC$). Then we have $\deg(w_4)=h_1h_2$, $\deg(w_5)=h_2h_3$, $\deg(w_6)=h_1h_2h_3$ and $\deg(w_7)=h_1h_3$. 

This grading is, up to isomorphism, independent of the choice of generators of $H$, because of \cite[Theorem 4.19]{EKmon}. Its support is, precisely, the subgroup $H$.

Any grading on $\cC$ is a grading on $(\cC,\{...\})$, so this gives a grading of $(\cC,\{...\})$, also denoted by  $\Gamma^\cC(G,H)$.
\end{example}

Given a grading $\Gamma$ of $(\cC,\{...\})$ by an abelian group $G$, 
and given an order two element $h\in G$, the \emph{shift} of 
$\Gamma$ by $h$ is the new grading $\Gamma^{[h]}$ with $\deg_{\Gamma^{[h]}}(x)=h\deg_{\Gamma}(x)$. It is clear that $\Gamma^{[h]}$ is a grading of $(\cC,\{...\})$ too.

\begin{example}\label{ex:gGHK}
Let $H$ be an elementary $2$-subgroup of rank $4$  of an abelian group $G$, and let $K$ be a subgroup of $H$ of index two (and hence $K$ is isomorphic to $\left(\ZZ/2\right)^3$). Take an element 
$h\in H\setminus K$, and consider the shift 
$\bigl(\Gamma^\cC(G,K)\bigr)^{[h]}$.

Item (i) of Lemma \ref{le:orbits_83} shows that, up to isomorphism, this grading is independent of $h$, and will be denoted by 
$\Gamma^\cC(G,H,K)$. Its support is $H\setminus K$.
\end{example}

We are ready to classify the gradings of the $3$-fold cross product 
$(\cC,X_1^\cC)$ or, equivalently, of the triple system $(\cC,\{...\})$, up to isomorphism.

\begin{theorem}\label{th:gr_83}
Let $\cC$ be the Cayley algebra over an algebraically closed field $\FF$, let $G$ be an abelian group and let $\Gamma:\cC=\bigoplus_{g\in G}\cC_g$ be a grading of $(\cC,\{...\})$. Then $\Gamma$ is isomorphic to one of the following gradings:
\begin{itemize}
\item $\Gamma^\cC(G,\alpha)$, for a group homomorphism 
$\alpha:\ZZ^3\rightarrow G$.
\item $\Gamma^\cC(G,H)$ for an elementary $2$-subgroup $H$ of rank $3$.
\item $\Gamma^\cC(G,H,K)$ for an elementary $2$-subgroup $H$ of rank $4$ and a subgroup $K$ of $H$ of index $2$.
\end{itemize}
Gradings on different items are not isomorphic, and 
\begin{itemize} 
\item Two gradings $\Gamma^\cC(G,\alpha)$ and $\Gamma^\cC(G,\alpha')$ are isomorphic if and only if there is an element $\omega$ in the Weyl group 
$W(\Gamma^\cC\subcartan)$ such that $\alpha'=\alpha\omega$.
\item Two gradings $\Gamma^\cC(G,H)$ and $\Gamma^\cC(G,H')$ are isomorphic if and only if $H=H'$.
\item Two gradings $\Gamma^\cC(G,H,K)$ and $\Gamma^\cC(G,H',K')$ are isomorphic if and only if $H=H'$ and $K=K'$.
\end{itemize}
\end{theorem}
\begin{proof}
Let $\Gamma:\cC=\bigoplus_{g\in G}\cC_g$ be a grading by the abelian group $G$ of $(\cC,\{...\})$. 

Assume first that there is a nonzero homogeneous element 
$x\in\cC_g$ with $\nup(x)=0$. As $\AAut(\cC,\{...\})\subseteq \Os^+(\cC,\nup)$, $\nup(\cC_{g_1},\cC_{g_2})=0$ unless $g_1g_2=e$, 
so there is an element $y\in\cC_{g^{-1}}$ such that $\nup(x,y)=1$ and $\nup(y)=0$. By Lemma \ref{le:orbits_83}(iii), we may assume that $e_1$ and $e_2$ are homogeneous, and $\deg(e_1)=g$, $\deg(e_2)=g^{-1}$. But then 
$\{e_1\cC e_2\}=e_1\cC e_2=U=\FF u_1+\FF u_2+\FF u_3$, and 
$\{e_2\cC e_1\}=V=\FF v_1+\FF v_2+\FF v_3$, are graded subspaces 
of $\cC$. In particular, we can take a homogeneous basis $\{u_1',u_2',u_3'\}$ of $U$ and, multiplying by a nonzero scalar $u_3'$
if needed, with $\nup(u_1',u_2'u_3')=1$. There is then an automorphism of $\cC$ that fixes $e_1$ and $e_2$ and takes $u_i$ 
to $u_i'$ for $i=1,2,3$, so we may assume that the $u_i$'s are homogeneous too, and so are the $v_i$'s (for instance, 
$v_3=\{u_1e_1u_2\}$). Therefore, $\Gamma$ is a coarsening of the Cartan grading, and hence isomorphic to $\Gamma^\cC(G,\alpha)$ for 
a group homomorphism $\alpha:\ZZ^3\rightarrow G$.

Otherwise, all homogeneous components are one-dimensional and not
isotropic. As $\{x,x,x\}=\nup(x)x$ for all $x\in\cC$, we conclude that the support of $\Gamma$ generates a $2$-elementary abelian subgroup. There are two possibilities:
\begin{itemize}
\item If the neutral element $e$ of $G$ is in the support, by Lemma 
\ref{le:orbits_83}(i) we may assume that the unity $1$ of $\cC$ is 
homogeneous of degree $e$: $1\in\cC_e$. But then, as $xy=\{x,1,y\}$, 
it follows that 
$\Gamma$ is a $G$-grading of $\cC$, with one-dimensional 
nonisotropic homogeneous components. This gives the second 
possibility (see \cite{EldOctonions} or \cite[Theorem 4.21]{EKmon}).

\item Otherwise, again by Lemma \ref{le:orbits_83}(i) we may assume 
that the unity $1$ of $\cC$ is homogeneous: $1\in\cC_g$, for an order 
$2$ element $g\in G$. Then in the shift $\Gamma^{[g]}$, $1$ is 
homogeneous of degree $g^2=e$, and we are in the situation of the 
previous item. 
If $K$ is the support of $\Gamma^{[g]}$, then the subgroup generated by $g$ and $K$ is $2$-elementary of rank $4$, and 
$\Gamma=\left(\Gamma^{[g]}\right)^{[g]}$ is, up to isomorphism, the grading $\Gamma^\cC(G,H,K)$.
\end{itemize}

Now it is clear that gradings in different items are not isomorphic. The support of $\Gamma^\cC(G,H)$ is $H$, and of $\Gamma^\cC(G,H,K)$ is $H\setminus K$. The isomorphism condition for gradings of type
$\Gamma^\cC(G,\alpha)$
 follows from \cite[Proposition 4.22]{EKmon}.
\end{proof}

The homogeneous components of the gradings $\Gamma^\cC(G,H)$ and $\Gamma^\cC(G,H,K)$ in Theorem \ref{th:gr_83} are the
 subspaces spanned by the elements of a CD-basis (Equation
 \eqref{eq:CDbasis}, and hence they are all equivalent to the grading 
$\Gamma^\cC\subCD$ over the grading group $\left(\ZZ/2\right)^4$ with
\[
\begin{aligned}
\deg(1)&=(\bar 1,\bar 0,\bar 0,\bar 0),&
\deg(w_1)&=(\bar 1,\bar 1,\bar 0,\bar 0),\\
\deg(w_2)&=(\bar 1,\bar 0,\bar 1,\bar 0),&
\deg(w_3)&=(\bar 1,\bar 0,\bar 0,\bar 1).
\end{aligned}
\]
(All the other homogeneous components are  determined from these ones.)

\begin{corollary}\label{co:finegr_83}
Up to equivalence, the only fine gradings of $(\cC,\{...\})$ are 
$\Gamma^\cC\subcartan$ and $\Gamma^\cC\subCD$, with universal
groups $\ZZ^3$ and $\left(\ZZ/2\right)^4$.
\end{corollary}

Any element of the Weyl group of $\Gamma^\cC\subCD$ permutes 
its support and gives an automorphism of the universal grading group.
 Consider $\ZZ/2$ as the field $\FF_2$ of elements, then
 $W(\Gamma^\cC\subCD)$ embeds in 
$\{\gamma\in\GL(\FF_2^4)\mid \varphi(1\times\FF_2^3)
\subseteq 1\times\FF_2^3\}$, which is identified with the affine 
group $\Aff(3,\FF_2)$. Lemma \ref{le:orbits_83}(i) shows that 
$W(\Gamma^\cC\subCD)$ acts transitively on the support. Also, 
the grading $\Gamma^\cC\subCD$ is equivalent to the fine 
$\bigl(\ZZ/2\bigr)^{\hspace*{-2pt}3}$-grading on $\cC$, with Weyl group 
$\GL_3(\FF_2)$. As $\Aut(\cC)$ is a subgroup of $\Aut(\cC,\{...\})$, 
it follows that $\GL_3(\FF_2)$ is contained in 
$W(\Gamma^\cC\subCD)\leq \Aff(3,\FF_2)$. It turns out that the 
Weyl group $W(\Gamma^\cC\subCD)$ is the whole affine group 
$\Aff(3,\FF_2)$.

The corollary above and the computation of the Weyl groups 
have been considered independently in \cite{KantorSystems}, 
devoted to the classification of the
fine gradings on certain Kantor systems attached to Hurwitz algebras. 
One such triple system corresponds to our
$3$-fold cross product $(\cC,X_1^\cC)$.

\begin{corollary}\label{co:quasitori_Spin7}
Let $\mathbf{Q}$ be a quasitorus (i.e., diagonalizable) subgroup scheme of $\Spins(\cC_0,-\nup)$. Then either:
\begin{itemize}
\item $\mathbf{Q}$ is contained in a maximal torus, and hence conjugate to $\Diags(\Gamma^\cC\subcartan)$, or
\item $\mathbf{Q}$ is conjugate to $\Diags(\Gamma^\cC\subCD)$, which is isomorphic to $\bmu_2^4$. 
\end{itemize}
\end{corollary}

\begin{remark}\label{re:Spin7SO7}
The group scheme $\Spins(\cC_0,-\nup)$ is the simply connected group 
of type $B_3$. In contrast, the corresponding adjoint group, that is, the 
special orthogonal group scheme
$\Os^+(\cC_0,-\nup)$, contains four
maximal quasitori, up to conjugation, which are isomorphic to 
$\Gsm^3$ (a maximal torus), $\Gsm^2\times\bmu_2^2$, $\Gsm\times\bmu_2^4$, and 
$\bmu_2^6$ (see \cite[Theorem 3.67]{EKmon}).
\end{remark}

\bigskip\bigskip




\begin{thebibliography}{KMRT98}

\bibitem[AOCMpr]{KantorSystems}
D.~Aranda-Orna and A.S.~C\'ordova-Mart{\'\i}nez, \emph{Fine 
gradings on Kantor systems of Hurwitz type}, in preparation.

\bibitem[BG67]{BG}
 R.B.~Brown and A.~Gray, \emph{Vector cross products}, 
Comment. Math. Helv. \textbf{42} (1967), 222-236.

\bibitem[CRE16]{CastilloElduque}
A.~Castillo-Ram{\'\i}rez and A.~Elduque, \emph{Some special
 features of Cayley algebras, and $G_2$, in low characteristics}, 
J.~Pure Appl. Algebra \textbf{220} (2016), no.~3, 1188-1205.

\bibitem[Eck43]{Eckmann} 
B.~Eckmann, \emph{Stetige Lösungen linearer Gleichungssysteme} Comment. Math. Helv. \textbf{15} (1943), 318--339

\bibitem[Eld96]{Korean} 
A.~Elduque, \emph{On a class of ternary
composition algebras},  J.~Korean Math. Soc. \textbf{33} (1996), no.~1, 183-203.

\bibitem[Eld98]{EldOctonions}
A.~Elduque, \emph{Gradings on octonions},  J.~Algebra \textbf{207} (1998), no.~1, 342-354.


\bibitem[Eld00]{Eld00}
A.~Elduque, \emph{On triality and automorphisms and derivations of composition algebras}, Linear Algebra Appl. \textbf{314} (2000), 
no.~1-3, 49-74.

\bibitem[Eld04]{Eld_Talk}
A.~Elduque, \emph{Vector cross products}, \texttt{http://personal.unizar.es/elduque/Talks/} \texttt{crossproducts.pdf}, 2004.


\bibitem[EKO05]{EKO} 
A.~Elduque, N.~Kamiya, and S.~Okubo, 
\emph{$(-1,-1)$-balanced Freudenthal Kantor triple systems and 
noncommutative Jordan algebras}, 
J.~Algebra  \textbf{294} (2005), no.~1, 19-40.

\bibitem[EK13]{EKmon} 
A.~Elduque and M.~Kochetov, \emph{Gradings
on simple Lie algebras}, Mathematical Surveys and Monographs
\textbf{189}, American Mathematical Society, Providence, RI;
Atlantic Association for Research in the Mathematical Sciences
(AARMS), Halifax, NS, 2013.

\bibitem[KO03]{KO} 
N.~Kamiya and S.~Okubo, \emph{Construction of Lie superalgebras
 $D(2,1;\alpha)$, $G(3)$ and $F(4)$ from some triple systems}, 
Proc. Edinb. Math. Soc. (2) \textbf{46} (2003), no.~1, 87-98.

\bibitem[Kar05]{Karigiannis} 
S.~Karigiannis, \emph{Deformations of $G_2$ and $\textup{Spin}(7)$ structures}, Canad. J. Math. \textbf{57} (2005), no.~5, 1012-1055.

\bibitem[KMRT98]{KMRT} 
M.-A.~Knus, A.~Merkurjev, M.~Rost, and
J.-P.~Tignol, \emph{The book of involutions}, American
Mathematical Society Colloquium Publications, vol.~44, American
Mathematical Society, Providence, RI, 1998, With a preface in
French by J. Tits. 

\bibitem[Sha90a]{Shaw1}
R.~Shaw, \emph{Ternary composition algebras I. Structure theorems: definite and neutral signatures},
Proc. Roy. Soc. London Ser. A \textbf{431} (1990), no.~1881, 1--19.

\bibitem[Sha90b]{Shaw2}
R.~Shaw, \emph{Ternary composition algebras II. Automorphism groups and subgroups},
Proc. Roy. Soc. London Ser. A \textbf{431} (1990), no.~1881, 21--36.

\bibitem[SV00]{SpringerVeldkamp}
T.A.~Springer and F.D.~Veldkamp, \emph{Octonions, Jordan Algebras and Exceptional Groups}, Springer Monogr. Math., Springer-Verlag, Berlin, 2000.

\bibitem[Wat79]{Waterhouse}
W.C.~Waterhouse, \emph{Introduction to affine group schemes},
Graduate Texts in Mathematics \textbf{66}. Springer-Verlag, 
New York-Berlin, 1979.

\bibitem[Whi63]{Whitehead}
G.W.~Whitehead, \emph{Note on cross-sections in Stiefel manifolds},
Comment. Math. Helv. \textbf{37} (1962/63), 239--240.

\end{thebibliography}
\end{document}